\documentclass[12pt]{extarticle}
\usepackage{amssymb}
\usepackage{amsfonts}
\usepackage{times}
\usepackage{mathptmx}
\usepackage{amsmath}
\usepackage[usenames]{color}
\usepackage{mathrsfs}
\usepackage{amsfonts}
\usepackage{amssymb,amsmath}
\usepackage{cite}
\usepackage{cases}
\usepackage{amsthm}
\usepackage{comment}
\usepackage{enumitem} 
\def\deg{{\rm deg}}

\def\cl{\centerline}

\def\a{\alpha}

\def\l{\lambda}
\def\p{\partial}
\def\vs{\vspace*}
\def\B{\mathfrak{B}}

\def\C{\mathbb{C}}

\def\vs{\vspace*}

\def\g1{\mathfrak{g}}

\def\gg{[\cdot_{\l}\cdot]}
\def\nn{\{\cdot_{\l}\cdot\}}
\def\g{\mathfrak{G}}
\def\n{\mathfrak{N}}

\def\a{\mathfrak{A}}

\numberwithin{equation}{section}
\newtheorem{theo}{Theorem}[section]
\newtheorem{defi}[theo]{Definition}
\newtheorem{coro}[theo]{Corollary}
\newtheorem{lemm}[theo]{Lemma}
\newtheorem{prop}[theo]{Proposition}

\newtheorem{ex}[theo]{Example}

\newtheorem{re}[theo]{Remark}

\begin{document}
	\begin{center}
		{\bf\large  Post-Lie conformal algebra structures on Lie conformal algebras }
		\footnote{
			Corresponding author{$^\ast$}: Lamei Yuan, lmyuan@hit.edu.cn 
		}
	\end{center}

	\cl{Lamei Yuan{$^\ast$}, Yuhui Tan}
	
	\cl{\small School of Mathematics, Harbin Institute of Technology, Harbin
		150001, China}
	\cl{\small E-mail: lmyuan@hit.edu.cn, 26B312001@stu.hit.edu.cn}
	\vs{6pt}

	\vs{8pt}
	
	\small\footnotesize
	\parskip .005 truein
	\baselineskip 3pt \lineskip 3pt
	\noindent{{\bf Abstract:} In this paper, we introduce and study post-Lie conformal algebras (PLCAs),  a generalization of post-Lie algebras to conformal algebras. We establish an equivalence between PLCA structures and Rota-Baxter operators of weight 1 on Lie conformal algebras. We also show that every PLCA induces a new Lie conformal algebra and study PLCA structures on pairs of Lie conformal algebras. Finally, we classify all PLCA structures on two important classes of Lie conformal algebras: $\B(q)$ and $W(b)$, achieved through Rota–Baxter operators of weight 1.

		\vs{10pt}

		\vs{5pt}
		
		
		\parskip .001 truein\baselineskip 6pt \lineskip 6pt
		
		\section{Introduction}
		
		Lie conformal algebras, introduced by Kac in \cite{VK1,VK2}, play a fundamental role in conformal field theory, vertex algebras, and the study of infinite-dimensional Lie algebras (\cite{VK3}). These algebraic structures provide an axiomatic framework for the operator product expansion of chiral fields (\cite{BV2}) and have applications in Hamiltonian formalisms for nonlinear evolution equations (\cite{ID}). The structure theory, representation theory, and cohomology theory of finite Lie conformal algebras have been well-developed (see \cite{BV1,BA,AV2,SV1,SV2,AK}).
		

		Post-Lie algebras, introduced by Vallette in the study of Koszul duality of operads \cite{B}, generalize pre-Lie (or left-symmetric) algebras and arise in various areas of mathematics and physics, including differential geometry, numerical integration on manifolds \cite{HA}, and stochastic analysis \cite{YM,YF2,M}. A post-Lie algebra $(\mathfrak{A},[\cdot,\cdot],\triangleright)$ is a Lie algebra $(\mathfrak{A},[\cdot,\cdot])$ together with a product $\triangleright:\mathfrak{A}\times\mathfrak{A}\rightarrow\mathfrak{A}$ satisfying the compatibility conditions
		$$x\triangleright[y,z]=[x\triangleright y,z]+[y,x\triangleright z]$$
		and 
		$$[x,y]\triangleright z=a_{\triangleright}(x,y,z)-a_{\triangleright}(y,x,z),$$
		where $a_{\triangleright}(x,y,z)$ is the associator of the product $\triangleright$, defined by $$a_{{\triangleright}}(x,y,z)=x\triangleright(y\triangleright z)-(x\triangleright y)\triangleright z.$$
		Post-Lie algebras have been studied from different
		points of view including constructions of nonabelian generalized Lax pairs \cite{CL1}, Poincar\'e-Birkhoff-
		Witt type theorems \cite{V1,V2}, factorization theorems \cite{KI} and relations to post-Lie groups.
		
		The concept of Rota-Baxter operators on associative algebras originated in G. Baxter’s work on fluctuation theory in probability \cite{G} and was later developed by Rota \cite{R}, Atkinson \cite{A}, and Cartier \cite{C} through investigations of their connections to combinatorics. These operators have since found significant applications in diverse areas, including Connes-Kreimer’s algebraic approach to renormalization in perturbative quantum field theory \cite{AD}, operad splitting \cite{M2,COLX}, and double Poisson algebras \cite{MP,MV}.
		In the context of Lie algebras, Rota-Baxter operators arise as the operator form of the classical Yang-Baxter equation. Moreover, relative Rota-Baxter operators (or $\mathcal{O}$-operators) naturally induce pre-Lie or post-Lie algebra structures \cite{C,CL1}. Recently, their generalizations to conformal algebras have been extensively studied in \cite{YF,HB1,HB2,Yuan1,Yuan2}.

		In this paper, we introduce and study post-Lie conformal algebras (PLCAs) as a natural conformal analogue of post-Lie algebras. Specifically, PLCAs provide a unifying framework encompassing Lie conformal algebras and left-symmetric conformal algebras, thereby extending the classical notion of post-Lie algebras to the conformal setting. We show that every PLCA induces a new Lie conformal algebra structure and give a systematic characterization of PLCA structures arising from pairs of Lie conformal algebras.
		Moreover, for any PLCA $(V, \circ_{\lambda}, [\cdot_{\lambda}\cdot])$, the left multiplication operator defined by $L(x)_{\lambda} y = x \circ_{\lambda} y$
		induces a conformal derivation of the underlying Lie conformal algebra $(V, [\cdot_{\lambda}\cdot])$. This observation reveals a close relationship between PLCAs and Rota--Baxter operators of weight 1 on Lie conformal algebras (see Corollary~\ref{cor3.15}), which serves as a fundamental tool in the classification of all possible PLCA structures on a given Lie conformal algebra.

	The paper is organized as follows. In Section~2, we recall basic definitions and preliminary results on conformal algebras and Rota--Baxter operators. In Section~3, we introduce post-Lie conformal algebras (PLCAs) and investigate their fundamental properties. Sections~4 and~5 are devoted to the classification of all PLCA structures on the Lie conformal algebras $W(b)$ and $\B(q)$, respectively, where in the latter case we restrict our attention to the parameters $q \neq 0$ and $q \notin \mathbb{Z}_{-}$.

	Throughout, all vector spaces and their tensor products are over the complex field $\mathbb{C}$. We denote by $\mathbb{N}$, $\mathbb{Z}_{-}$ and $\mathbb{Z}_+$ the sets of all positive, negative and non-negative integers, respectively. For an algebra $\mathfrak{A}$, we use $Z(\mathfrak{A})$ to denote its center.

		\section{Preliminaries}
		In this section, we review the definitions of Lie conformal algebras, their modules, conformal derivations, and Rota-Baxter operators, as well as the notion of left-symmetric conformal algebras.

		\begin{defi}{\rm 
				A conformal algebra $R$ is a $\C[\p]$-module endowed with a $\C$-bilinear map $R\times R\rightarrow R[\l]$, denoted by $a\times b\mapsto a_{\l}b$, satisfying conformal sesquilinearity ($a,b\in R$): 
				\begin{align*}
					(\p a)_{\l}b=-\l(a_{\l}b),~a_{\l}(\p b)=(\p+\l)(a_{\l}b).
			\end{align*}}
	\end{defi} 
	
	A conformal algebra is called {\it finite} if it is finitely generated as a $\C[\p]$-module. The {\it rank} of a conformal algebra $R$ is its rank as a $\C[\p]$-module.

	\begin{defi}{\rm A $\mathbb{C}[\partial]$-module $R$ is called a \textbf{Lie conformal algebra} if it is equipped with a family of $\mathbb{C}$-bilinear products $-_{(n)}-$: $R \times R \rightarrow R$  satisfying the following axioms:
			\begin{itemize}
				\item [(L0)] (Locality)For any $a, b \in R$, there exists a positive integer $N$ such that $a_{(n)} b = 0$ for all $n \geq N$.
				\item [(L1)] (Derivation property) For any  $a, b \in R$  and  $n \in \mathbb{Z}_{+}$, $(\partial a)_{(n)} b=-n a_{(n-1)} b$, $a_{(n)}\p b=\p(a_{(n)}b)+na_{(n-1)}b$;
				\item [(L2)] (Skew-symmetry) For any  $a, b \in R$  and  $n \in \mathbb{Z}_{+}$, $a_{(n)} b=-\sum_{j=0}^{\infty}(-1)^{j+n}  \frac{\partial^{j}}{j!}\left(b_{(n+j)} a\right)$;
				\item [(L3)]  (Jacobi identity) For any  $a, b, c \in R$, $m, n \in \mathbb{Z}_{+}$,
				$$a_{(m)}\left(b_{(n)} c\right)=\sum_{j=0}^{m}\binom{m}{j}\left(a_{(j)} b\right)_{(m+n-j)} c+b_{(n)}\left(a_{(m)} c\right).$$
		\end{itemize}}
	\end{defi}		

	The $n$-products can be encoded via the $\lambda$-bracket $[\cdot_\lambda \cdot]: R \otimes R \rightarrow  \mathbb{C}[\lambda] \otimes R$ defined by
	$$\left[a_{\lambda} b\right]=\sum_{n=0}^{\infty} \frac{\lambda^{n}}{n!} a_{(n)} b,\quad \forall a, b \in R$$
Equivalently, $R$ is a Lie conformal algebra if and only if the $\lambda$-bracket satisfies the axioms listed below (for all $a, b, c \in R$):

	\begin{itemize}
		\item [(L0$_{\l}$)] (Polynomiality)  $a \otimes b \mapsto[a_{\l} b]$  determines a $\C$-bilinear map  $R \otimes R \rightarrow \C[\lambda] \otimes R$;
		\item [(L1$_{\l}$)] (Sesquilinearity)  $[(\partial a)_{\l} b]=-\l[a_{\lambda} b],~[a_{\l}\p b]=(\p+\l)[a_{\l}b]$;
		\item [(L2$_{\l}$)] (Skew-symmetry)  $[a_{\lambda} b]=-[b_{-\partial-\lambda} a]$ ;
		\item [(L3$_{\l}$)] (Jacobi identity)  $[a_{\lambda}[b_{\mu} c]]-[b_{\mu}[a_{\lambda} c]]=[[a_{\lambda} b]_{\lambda+\mu} c]$.
	\end{itemize}	
	
	\begin{ex} The
		Virasoro conformal algebra ${\rm Vir}$  is a free $\C[\p]$-module generated by the symbol $L$, such that $$[L_{\l}L]=(\p+2\l)L.$$ 
	\end{ex}
	
	\begin{ex}(see \cite{AK})
		Let \(\mathfrak{g}\) be a Lie algebra with Lie bracket \([\cdot , \cdot]\), and let \(\operatorname{Cur\mathfrak{g}} := \mathbb{C}[\partial] \otimes \mathfrak{g}\) be the free \(\mathbb{C}[\partial]\)-module. Then \(\operatorname{Cur\mathfrak{g}}\) is a Lie conformal algebra, called current Lie conformal algebra, with the \(\lambda\)-bracket given by:
		\[
		\big(f(\partial) \otimes a\big)_\lambda \big(g(\partial) \otimes b\big) := f(-\lambda)g(\lambda + \partial) \otimes [a, b], \quad \forall a, \, b \in \mathfrak{g},~f(\partial), \, g(\partial) \in \mathbb{C}[\partial].
		\]
	\end{ex}
	Let $V$ be a $\mathbb{C}[\partial]$-module, and let $v(\lambda) = \sum\limits_j v_j \lambda^j \in \mathbb{C}[\partial,\lambda] \otimes_{\mathbb{C}[\partial]} V$, where $v_j \in V$. Define $\langle v \rangle$ to be the $\mathbb{C}[\partial]$-submodule of $V$ spanned by all coefficients $v_j$ of $v(\lambda)$. The following lemma will be used in the sequel (see \cite{AK}).
	
	\begin{lemm}
		Let $p(\lambda) = \sum\limits_{i=0}^{m} p_i \lambda^i \in \C[\p,\lambda]$, $p_{i}\in\C[\p]$ be a polynomial whose leading coefficient $p_m$ does not depend on $\partial$ (i.e., $p_m \in \mathbb{C}^*$). Then for any $v \in \C[\p,\lambda]$, we have $\langle p(\lambda) v \rangle = \langle v \rangle$.
	\end{lemm}
	
	A conformal {\it subalgebra} $S$ of a Lie conformal algebra $R$ is itself a Lie conformal algebra with the operations inherited from $R$ or equivalently, 
	$\langle[a_{\lambda} b] \rangle\in S$ for all $a,b \in S.$ A $\C[\p]$-submodule $I$ of $R$ is {\it an ideal of $R$} if  $\langle [I_{\lambda} R]\rangle\subseteq I$ and because of (L3$_{\l}$) we get 
	$\langle [R_{\lambda} I]\rangle\subseteq I.$
	
	Below we define the product of two ideals of a Lie conformal Lie algebra (see \cite{AK}).
	
	\begin{defi} Let $R$ be a Lie conformal algebra, and let $I$ and $J$ be ideals of $R$. The \emph{bracket} $[I\cdot J]$ is defined as the subspace of $R$ spanned by coefficients of all products $[i_{\lambda} j]$, where $i \in I$, $j \in J$. This subspace forms a $\mathbb{C}[\partial]$-module due to axiom ($L2_{\lambda}$) and constitutes an ideal of $R$ due to axiom ($L4_{\lambda}$). In other words, $[I\cdot J] = \langle [I_{\l} J] \rangle$
	\end{defi}

	For a given Lie conformal algebra $R$ we define lower central and derived series, respectively:  
	$$R^1 = R, \quad R^{k+1} = [R^k \cdot R] \quad \mbox{and} \quad 
	R^{[1]} = R, \quad R^{[k+1]} = [R^{[k]} \cdot R^{[k]}], \quad k \geq 1,$$
	
	\begin{defi} A conformal Lie algebra $R$ is called nilpotent (respectively, solvable) if there exists $n\in \mathbb{N}$ such that $R^n=0$ (respectively, $R^{[n]}=0).$ 
	\end{defi}

	\begin{defi}  {\rm (\cite{YF})
			A left-symmetric conformal algebra $R$ is a conformal algebra equipped with a $\C$-bilinear map $\cdot_{\l}\cdot : R \times R \rightarrow \C[\lambda] \otimes R$, satisfying
			\begin{align*}
				(a_{\l}b)_{\l+\mu}c-a_{\l}(b_{\mu}c)=(b_{\mu}a)_{\l+\mu}c-b_{\mu}(a_{\l}c), ~\mbox{for~ all}~ a, b, c\in R.
		\end{align*}}
	\end{defi}
	
	\begin{defi}{\rm
			A module $M$ over a Lie conformal algebra $R$ is a $\C[\partial]$-module endowed with a $\C$-bilinear map $R\times M$ $\rightarrow M[\lambda]$, $(a,v)\mapsto a_{\lambda}v$, satisfying the following axioms $(\forall a,b\in R,~v\in M)$
			\begin{itemize}
				\item[(1)] $(\partial a)_{\lambda}v=-\l a_{\l}v,~a_{\lambda}(\partial v)=(\partial+\l) a_{\l}v,$ 
				\item[(2)] $[a_{\l}b]_{\l+\mu}v=a_{\l}(b_{\mu}v)-b_{\mu}(a_{\l}v).$
		\end{itemize}}
	\end{defi}
	An $R$-module $M$ is called {\it finite} if it is finitely generated as a $\C[\p]$-module.
	\begin{defi}{\rm
			Given two $\C[\p]$-modules $M$ and $N$, a conformal linear map from $M$ to $N$ is a $\C$-linear map $\varphi:M\rightarrow\C[\l]\otimes_{\C[\p]}N$, denoted by $\varphi_{\l}:M\rightarrow N$, such that $$[\p,\varphi_{\l}]=-\l \varphi_{\l}, ~~\mbox{equivalently}, ~~ \p^{N}\varphi_{\l}-\varphi_{\l}\p^{M}=-\l \varphi_{\l}.$$ }
	\end{defi}
	The vector space of all conformal linear map from $M$ to $N$ is denoted by $Chom(M,N)$ and it can be made into a $\C[\p]$-module via 
	$$(\p \varphi)_{\l}v=-\l \varphi_{\l}v, ~~~\mbox{for}~\varphi\in Chom(M,N), ~ v\in M.$$
	We will write $Cend({M})$ for $Chom(M,M)$.

	\begin{defi}{\rm Let 
			$R$ be a Lie conformal algebra. A conformal linear map $d_\l:R\rightarrow R$ is called a conformal derivation of $R$ if 
			\begin{align}\label{eq2.1}
				d_{\l}([x_{\mu}y])=[(d_{\l}x)_{\l+\mu}y]+[x_{\mu}(d_{\l}y)], ~~ \forall ~~x,y\in R.  
		\end{align}}
	\end{defi}
	The space of all conformal derivations of $R$ is denoted by ${\rm CDer} (R)$.
	Let $R$ be a Lie conformal algebra. Every $a\in R$ determines a conformal linear map $({\rm ad\,} a)_{\l}: R\rightarrow R$ by $$({\rm ad\,} a)_\l b=[a_\l b],~~ b\in R. $$
	By Jacobi identity, $({\rm ad\,} a)_{\l}$ is a conformal derivation of $R$ for every $a\in R$. All derivations of this kind are called {\it inner}. The set of all inner derivations is denoted by  ${\rm CInn}(R)$.                              
	
	\begin{defi}{\rm 
			A Lie conformal algebra $R$ is called complete if it has trivial center and ${\rm CDer} (R)={\rm CInn}(R)$.}
	\end{defi}
	
	\begin{defi}
		Let $R$ be a Lie conformal algebra and let $h \in \mathbb{C}$. 
		A \emph{Rota-Baxter operator of weight $h$} on $R$ is a $\mathbb{C}[\partial]$-linear map $T : R \rightarrow R$
		satisfying the identity
		\begin{equation*}
			[T(a)_\lambda T(b)] 
			= T\big( [T(a)_\lambda b] + [a_\lambda T(b)] + h [a_\lambda b] \big),
			\quad \forall\, a,b \in R.
		\end{equation*}
	\end{defi}
	
	\begin{re}
		If $T$ is a Rota--Baxter operator of weight $h \neq 0$ on a Lie conformal algebra $R$,
		then $h^{-1}T$ is a Rota--Baxter operator of weight $1$.
		Hence the study of weighted Rota--Baxter operators reduces to the weight $1$ case.
		Consequently, it suffices to investigate Rota--Baxter operators of weight $1$, i.e., operators satisfying
		\begin{equation}\label{eq2.2}
			[T(a)_{\lambda} T(b)] = 
			T\bigl( [T(a)_{\lambda} b] + [a_{\lambda} T(b)] + [a_{\lambda} b] \bigr),
			\qquad \forall\, a,b \in R.
		\end{equation}
	Note that $T = 0$ and $T' = -\mathrm{id}_R$ are Rota-Baxter operators of weight $1$ on $R$. Throughout this paper, we will refer to these as the {trivial} Rota-Baxter operators of weight $1$.
	\end{re}	
	
	\begin{re}\label{re2.15}
		Let $R$ be a Lie conformal algebra and $I$ an ideal of $R$. Suppose $T$ is a Rota-Baxter operator of weight 1 on $R$. It is straightforward to verify that $T$ naturally induces a Rota-Baxter operator  of weight 1 on the quotient Lie conformal algebra $R/I$.
	\end{re}	
	
	\begin{lemm}\label{lem2.16}({\cite{L}}) If $T$ is a Rota-Baxter operator of weight $1$ on a Lie conformal algebra $R$, then the operator $T' = -T - \mathrm{id}_R$ is also a Rota-Baxter operator of weight $1$ on $R$.
	\end{lemm}
	\section{Post-Lie conformal algebra structures}
	
	Post-Lie algebras, which combine the structures of Lie algebras and left-symmetric algebras, have emerged as a fundamental object in algebra and mathematical physics. In this section, we introduce their conformal analogue—post-Lie conformal algebras (PLCAs)—which naturally extend these ideas to the framework of conformal algebras.
	
	\subsection{PLCA structures on LCAs}
	
	\begin{defi}
		Let $V$ be a $\mathbb{C}[\partial]$-module equipped with two $\lambda$-products
		$x \circ_\lambda y$ and $[x_\lambda y]$. The triple
		$(V,\circ_\lambda,[\cdot_\lambda\cdot])$ is called a
		\emph{post-Lie conformal algebra} (PLCA) if
		$(V,[\cdot_\lambda\cdot])$ is a Lie conformal algebra, the $\lambda$-product
		$\circ_\lambda$ is conformal sesquilinear, and the following compatibility
		conditions hold for all $x,y,z \in V$:
		\begin{align}
			x \circ_\lambda [y_\mu z]
			&= \big[(x \circ_\lambda y)_{\lambda+\mu} z\big]
			+ \big[y_\mu (x \circ_\lambda z)\big],
			\label{eq3.1}\\
			[x_\lambda y] \circ_{\lambda+\mu} z
			&= x \circ_\lambda (y \circ_\mu z)
			- y \circ_\mu (x \circ_\lambda z)
			- (x \circ_\lambda y) \circ_{\lambda+\mu} z
			+ (y \circ_\mu x) \circ_{\lambda+\mu} z.
			\label{eq3.2}
		\end{align}
	\end{defi}
	
	Let $(V,\circ_{\lambda},[\cdot_{\lambda}\cdot])$ be a PLCA. For each $x \in V$,
	define the \emph{left multiplication operator} $L(x)$ by
	$L(x)_{\lambda} y = x \circ_{\lambda} y, ~ \forall\, y \in V.$
	Then $L(x)$ is a conformal derivation of the Lie conformal algebra
	$(V,[\cdot_{\lambda}\cdot])$.
	
	\begin{defi}
		Let $(V,[\cdot_{\lambda}\cdot])$ be a Lie conformal algebra. A
		\emph{PLCA structure} on $V$ is a $\lambda$-product $\circ_{\lambda}$ on $V$
		such that $(V,\circ_{\lambda},[\cdot_{\lambda}\cdot])$ forms a post-Lie
		conformal algebra.
	\end{defi}

	\begin{ex}
		Given a Lie conformal algebra $(V,[\cdot_{\lambda}\cdot])$, there are two natural PLCA structures on $V$, defined by 
	$$x\circ_{\l}y=0, ~~\mbox{and} ~~ x\circ_{\l}y=-[x_{\l}y], ~~\forall ~ x,y\in V. $$
	They are refereed as  {\it trivial} PLCA structures.
	\end{ex}
	
	\begin{ex}
		Let \(\operatorname{Cur}\mathfrak{g}\) be the current Lie conformal algebra associated with a Lie algebra \(\mathfrak{g}\). 
		Assume that \((\mathfrak{g}, [\,\cdot\,,\cdot\,], \circ)\) is a post-Lie algebra. 
		Then the $\l$-multiplication
		$$\big(f(\partial) \otimes a\big) \circ_{\lambda} \big(g(\partial) \otimes b\big)
		= f(-\lambda)\, g(\partial+\lambda) \otimes (a \circ b),
		\quad \forall ~a,b\in\mathfrak{g},~f(\partial),g(\partial)\in\mathbb{C}[\partial]$$
		defines a PLCA structure on \(\operatorname{Cur}\mathfrak{g}\).
	\end{ex}
	
	\begin{prop} 
		Let $(R,\circ_{\lambda},[\cdot_{\lambda}\cdot])$ be a PLCA. Then $(R,\{\cdot_{\lambda}\cdot\})$ forms a Lie conformal algebra, where the new $\l$-bracket $\{\cdot_{\lambda}\cdot\}$ is defined by 
		\begin{align}\label{asso}
			\{x_{\lambda}y\}=x\circ_{\lambda}y-y\circ_{-\lambda-\partial}x+[x_{\lambda}y], ~~\forall ~~ x,y\in R.
		\end{align}
	We call $(R,\{\cdot_{\lambda}\cdot\})$ is the Lie conformal algebra associated to $(R,\circ_{\l},\gg)$.
	\end{prop}
	
	\begin{proof}For all $x,y\in R$, we first verify conformal sesquilinearity:
		\begin{align*}
			\{(\partial x)_{\lambda}y\}&=(\partial x)\circ_{\lambda}y-y\circ_{-\lambda-\partial}(\partial x)+[(\partial x)_{\lambda}y]=(-\lambda)x\circ_{\lambda}y-(\partial-\lambda-\partial)y\circ_{-\lambda-\partial}x+(-\lambda)[x_{\lambda}y]=-\lambda\{x_{\lambda}y\}\\
			\{x_{\lambda}(\partial  y)\}&=x\circ_{\lambda}(\partial  y)-(\partial  y)\circ_{-\lambda-\partial}x+[x_{\lambda}(\partial  y)]=(\partial+\lambda)x\circ_{\lambda}y-(\lambda+\partial)y\circ_{-\lambda-\partial}x+(\partial+\lambda)[x_{\lambda}y]=(\partial+\lambda)\{x_{\lambda}y\}.
		\end{align*}
		Skew-symmetry follows from:
		\begin{align*}
			\{x_{\lambda}y\}+\{y_{-\lambda-\partial}x\}=x\circ_{\lambda}y-y\circ_{-\lambda-\partial}x+[x_{\lambda}y]
			+y\circ_{-\lambda-\partial}x-x\circ_{\lambda}y+[y_{-\lambda-\partial}x]
			=0.
		\end{align*}
		The Jacobi identity follows from a straightforward computation using the definitions and the fact that both $\circ_\lambda$ and $[\cdot_\lambda\cdot]$ satisfy their respective versions of the Jacobi identity in a PLCA:	
		\begin{align*}
			&\{\{x_{\lambda}y\}_{\lambda+\mu}z\}+\{y_{\mu}\{x_{\lambda}z\}\}-\{x_{\lambda}\{y_{\mu}z\}\}\\
			=&\{(x\circ_{\lambda}y)_{\lambda+\mu}z\}-\{(y\circ_{-\lambda-\partial}x)_{\lambda+\mu}z\}+\{[x_{\lambda}y]_{\lambda+\mu}z\}+\{y_{\mu}(x\circ_{\lambda}z)\}\\&-\{y_{\mu}(z\circ_{-\lambda-\partial}x)\}+\{y_{\mu}[x_{\lambda}z]\}
			-\{x_{\l}(y\circ_{\mu}z)\}+\{x_{\l}(z\circ_{-\mu-\partial}y)\}-\{x_{\l}[y_{\mu}z]\}\\
			=&[[x_{\lambda}y]_{\lambda+\mu}z]+[y_{\mu}[x_{\lambda}z]]-[x_{\l}[y_{\mu}z]]
			-x\circ_{\l}[y_{\mu}z]+[(x\circ_{\lambda}y)_{\lambda+\mu}z]+[y_{\mu}(x\circ_{\lambda}z)]\\
			&+y\circ_{\mu}[x_{\lambda}z]-[(y\circ_{-\lambda-\partial}x)_{\lambda+\mu}z]-[x_{\l}(y\circ_{\mu}z)]
			-z\circ_{-\lambda-\mu-\partial}[x_{\lambda}y]-[y_{\mu}(z\circ_{-\lambda-\partial}x)]+[x_{\l}(z\circ_{-\mu-\partial}y)]\\
			&+[x_{\lambda}y]\circ_{\lambda+\mu}z-x\circ_{\l}(y\circ_{\mu}z)+y\circ_{\mu}(x\circ_{\lambda}z)+(x\circ_{\lambda}y)\circ_{\lambda+\mu}z-(y\circ_{-\lambda-\partial}x)\circ_{\lambda+\mu}z\\
			&-[x_{\lambda}z]\circ_{-\mu-\partial}y+x\circ_{\l}(z\circ_{-\mu-\partial}y)-z\circ_{-\lambda-\mu-\partial}(x\circ_{\lambda}y)-(x\circ_{\lambda}z)\circ_{-\mu-\partial}y+(z\circ_{-\lambda-\partial}x)\circ_{-\mu-\partial}y\\				
			&+[y_{\mu}z]\circ_{-\l-\partial}x+z\circ_{-\lambda-\mu-\partial}(y\circ_{-\lambda-\partial}x)-y\circ_{\mu}(z\circ_{-\lambda-\partial}x)-(z\circ_{-\mu-\partial}y)\circ_{-\l-\partial}x+(y\circ_{\mu}z)\circ_{-\l-\partial}x\\
			=&0.
		\end{align*}
	\end{proof}
	
	\begin{prop} Let $(V,\circ_{\l},[\cdot_{\l}\cdot])$  be a PLCA and let $\{\cdot_{\l}\cdot\}$ be the associated Lie conformal algebra structure on $V$ defined in \eqref{asso}. Then the following identity holds:
		\begin{align*}
			\{x_{\l}y\}\circ_{\l+\mu}z=x\circ_{\lambda}(y\circ_{\mu} z)-y\circ_{\mu}(x\circ_{\lambda} z).
		\end{align*}
		Consequently, $V$ is a conformal $(V,\{\cdot_{\l}\cdot \})$-module with the $\l$-action given by $\circ_{\l}$. 
	\end{prop}
	
	\begin{proof}
		This follows immediately from \eqref{eq3.2} and the definition of $\{\cdot_{\l}\cdot\}$
		in \eqref{asso}.
	\end{proof}
	
	\subsection{PLCA structures on pairs of LCAs}
	
	Let $\g=(V,\gg)$ and $\n=(V,\nn)$ be two Lie conformal algebra structures on the same $\C[\p]$-module $V$; we call $(\g,\n)$ a \textit{pair of Lie conformal algebras} over $V$. In this subsection, when referring to such a pair $(\g,\n)$, we adopt the conventions that the $\l$-bracket in $\g$ is denoted by $[x_{\l}y]$, the $\l$-bracket in $\n$ is denoted by $\{x_{\l}y\}$, and the underlying $\C[\p]$-module for both $\g$ and $\n$ is denoted by $V$. Furthermore, the notation $(\operatorname{ad} a)_{\l}$ will exclusively refer to the adjoint operator in the Lie conformal algebra $\g$, i.e., $(\operatorname{ad} a)_{\l}(b)=[a_\l b]$ for all $b\in V$.
	
	\begin{defi}
			Let $(\g, \n)$ be a pair of Lie conformal algebras. A post-Lie conformal algebra structure on 
			$(\g, \n)$ is a conformal sesquilinear operation $\circ_{\l}:V\times V\rightarrow V[\l]$, satisfying the following axioms:
			\begin{align}
				\label{eq3.4}&\{x_{\lambda}y\}=x\circ_{\lambda}y-y\circ_{-\lambda-\partial}x+[x_{\lambda}y],\\
				\label{eq3.5}x\circ&_{\lambda}[y_{\mu} z]=[(x\circ_{\lambda} y)_{\lambda+\mu} z]+[y_{\mu}(x\circ_{\lambda} z)],\\
				\label{eq3.6}\{x_{\lambda} &y\}\circ_{\lambda+\mu} z=x\circ_{\lambda}(y\circ_{\mu} z)-y\circ_{\mu}(x\circ_{\lambda} z),
			\end{align}
			for all $x,y,z\in V$.
	\end{defi}
	
	Let $(\g, \n)$ be a pair of Lie conformal algebras over the same $\C[\p]$-module $V$, where $\g=(V,\gg)$ and $\n=(V,\nn)$. If $x\circ_{\l}y$ be a PLCA structure on $(\g,\n)$, then $\n$ is the Lie conformal algebra associated to $(V,\circ_{\l},\gg)$. Conversely, if $(V, \circ_{\l}, \gg)$ is a PLCA, then $\n$ is the associated Lie conformal algebra and $x\circ_{\l}y$ is a PLCA structure on the pair $(\g,\n)$.
	
	\begin{prop}\label{prop3.8} 
		Let $(\g, \n)$ be a pair of Lie conformal algebras over the same $\C[\p]$-module $V$. 
		\begin{itemize}
			\item [(1)]
			If $\g$ is abelian, then $x\circ_{\l}y$ induces a PLCA structure on $(\g, \n)$ if and only if $(V,\circ_{\l})$ forms a left-symmetric conformal algebra and the following identity holds: 
			\begin{align*}
				\{x_{\l}y\}=x\circ_{\l}y-y\circ_{-\l-\p}x,~~x,y\in V.
			\end{align*}
			\item [(2)]
			If $\n$ is abelian, then $x\circ_{\l}y$ induces a PLCA structure on $(\g, \n)$ if and only if the following identities hold: 
			\begin{align}
				\label{eq3.7}x\circ_{\l}y-y\circ_{-\l-\p}&x=-[x_{\l}y],\\
				\label{eq3.8}x\circ_{\lambda}(y\circ_{\mu} z)=y&\circ_{\mu}(x\circ_{\lambda} z),\\
				\label{eq3.9}(x\circ_{\l}z)\circ_{-\mu-\p}y=&(x\circ_{\l}y)\circ_{\l+\mu}z.
			\end{align}
		\end{itemize}
	\end{prop}
	\begin{proof}
		Part (1) can be obtained through straightforward calculation. Here we give the proof of Part (2). Since $\n$ is abelian, we can get \eqref{eq3.7} and \eqref{eq3.8} from \eqref{eq3.4} and \eqref{eq3.6}, respectively.  And we can get the following equality from \eqref{eq3.5} and \eqref{eq3.7}
		\begin{align*}
			x\circ_{\l}(y\circ_{\mu}z)-x\circ_{\l}(z\circ_{-\mu-\p}y)=(x\circ_{\l}y)\circ_{\l+\mu}z-z\circ_{-\l-\mu-\p}(x\circ_{\l}y)+y\circ_{\mu}(x\circ_{\l}z)-(x\circ_{\l}z)\circ_{-\mu-\p}y.
		\end{align*}
		By \eqref{eq3.8} we have $x\circ_{\l}(y\circ_{\mu}z)=y\circ_{\mu}(x\circ_{\l}z)$, $x\circ_{\l}(z\circ_{-\mu-\p}y)=z\circ_{-\l-\mu-\p}(x\circ_{\l}y)$, then we obtain Eq.\eqref{eq3.9}. Conversely, if Eq.s \eqref{eq3.7}--\eqref{eq3.9} are satisfied, we can conclude that $\circ_{\l}$ is a PLCA structure on $\g$ and the Lie conformal algebra associated $(V,\circ_{\l},[\cdot_{\l}\cdot])$ is abelian.
	\end{proof}
	
	\begin{prop}
		Let $\a=(V,\circ_{\l})$ be a PLCA structure on the pair $(\g,\n)$. If $\n$ is abelian, then $\g^{[3]}=0$. 
	\end{prop}
	
	\begin{proof}
		For $u_{1},v_{1},u_{2},v_{2}\in V$, by Proposition \ref{prop3.8} (2), we have 
		\begin{align*}
			(u_{1}\circ_{\l_{1}}v_{1})\circ_{-\mu-\p}(u_{2}\circ_{\l_{2}}v_{2})&=(u_{1}\circ_{\l_{1}}(u_{2}\circ_{\l_{2}}v_{2}))\circ_{\l_{1}+\mu}v_{1}=(u_{2}\circ_{\l_{2}}(u_{1}\circ_{\l_{1}}v_{2}))\circ_{\l_{1}+\mu}v_{1}\\&=(u_{2}\circ_{\l_{2}}v_{1})\circ_{-\l_{1}-\mu+\l_{2}-\p}(u_{1}\circ_{\l_{1}}v_{2})		=u_{1}\circ_{\l_{1}}((u_{2}\circ_{\l_{2}}v_{1})\circ_{-\mu+\l_{2}-\p}v_{2})\\&=u_{1}\circ_{\l_{1}}((u_{2}\circ_{\l_{2}}v_{2})\circ_{\mu}v_{1})=(u_{2}\circ_{\l_{2}}v_{2}) \circ_{\mu}(u_{1}\circ_{\l_{1}}v_{1}).
		\end{align*}
		Then we can conclude that for $x,y\in\a^{[2]}$, $x\circ_{\l}y=y\circ_{-\l-\p}x$,  Furthermore, we have $[x_{\l}y]=0$ by Proposition \ref{prop3.8} (2). Since $\g^{[2]}\subseteq \a^{[2]}$, we have $\g^{[3]}=0$.
	\end{proof}

	\begin{defi}{\rm 
			A commutative post-Lie conformal algebra (CPLCA) structure on a Lie conformal algebra $R$ is a  conformal sesquilinear product $\circ_{\lambda}$ satisfying 
			\begin{align}
				\label{eq3.10}x\circ_{\lambda}y=y\circ&_{-\lambda-\partial}x,\\
				 x\circ_{\lambda}[y_{\mu} z]=[(x\circ_{\lambda} y)&_{\lambda+\mu} z]+[y_{\mu}(x\circ_{\lambda} z)],\notag\\
				\label{eq3.11}[x_{\lambda} y]\circ_{\lambda+\mu}z=x\circ_{\lambda}(y&\circ_{\mu} z)-y\circ_{\mu}(x\circ_{\lambda} z),
			\end{align}
			for all $x, y, z\in R$.}
	\end{defi}
	There is always a {\bf trivial CPLCA structure} on  a Lie conformal algebra $R$, defined by $$x\circ_{\lambda}y=0,~~\mbox{ for~ all}~ x,y\in R.$$ However, determining whether a given Lie conformal algebra admits a non-trivial CPLCA structure is generally non-trivial. For abelian Lie conformal algebras, CPLCA structures correspond to commutative associative algebras, as illustrated below.
	
	\begin{prop}
		Let $(R, [\cdot_\lambda \cdot])$ be an abelian Lie conformal algebra with $[x_\lambda y] = 0$ for all $x,y\in R$, and let $\circ_\lambda$ be a CPLCA structure on $R$. Then $\circ_\lambda$ is both commutative and associative.
	\end{prop}
	
	\begin{proof}
		The commutativity follows directly from the CPLCA condition. For associativity, using relations \eqref{eq3.10} and \eqref{eq3.11}, we compute:
		\begin{align*}
			x \circ_\lambda (y \circ_\mu z)
			= x \circ_\lambda (z \circ_{-\mu-\partial} y)
			= z \circ_{-\mu-\partial-\lambda} (x \circ_\lambda y)
			= (x \circ_\lambda y) \circ_{\lambda+\mu} z,
		\end{align*}
		for all $x,y,z\in R$. This establishes the associativity of $\circ_\lambda$.
	\end{proof}

	\begin{prop}
		Let $R$ be a Lie conformal algebra with trivial center, and let
		$\phi \colon R \to R$ be a $\mathbb{C}[\partial]$-linear map. Define a
		$\lambda$-product $\circ_{\lambda}$ on $R$ by
		\[
		x \circ_{\lambda} y = [\phi(x)_{\lambda} y], \qquad \forall\, x, y \in R.
		\]
		Then $\circ_{\lambda}$ defines a CPLCA structure on $R$ if and only if
		$\phi$ is a homomorphism of Lie conformal algebras satisfying
		\begin{equation}\label{eq3.12}
			[\phi(x)_{\lambda} y] + [x_{\lambda} \phi(y)] = 0,
			\qquad \forall\, x, y \in R.
		\end{equation}
	\end{prop}
	
	\begin{proof}
		Assume that $\circ_{\lambda}$ defines a CPLCA structure on $R$. The
		commutativity condition implies that
		\[
		[\phi(x)_{\lambda} y]
		= x \circ_{\lambda} y
		= y \circ_{-\lambda-\partial} x
		= [\phi(y)_{-\lambda-\partial} x]
		= - [x_{\lambda} \phi(y)],
		\]
		where the last equality follows from the conformal skew-symmetry of the
		$\lambda$-bracket. This proves \eqref{eq3.12}.
		
		Moreover, using \eqref{eq3.11} together with the conformal Jacobi identity
		$(L4_{\lambda})$, we obtain
		\[
		[\phi(x)_{\lambda} \phi(y)] = \phi([x_{\lambda} y]),
		\]
		which shows that $\phi$ is a homomorphism of Lie conformal algebras.
		
		Conversely, if $\phi$ is a Lie conformal algebra homomorphism satisfying
		\eqref{eq3.12}, then a direct verification shows that the $\lambda$-product
		$\circ_{\lambda}$ satisfies all the axioms of a CPLCA.
	\end{proof}

	\subsection{Connection with Rota-Baxter operators of weight 1}
	
	\begin{lemm}\label{lem3.13}
		Let \(R\) be a Lie conformal algebra.
		For any \(d \in Z(\operatorname{CDer}(R))\), we have $\langle d_{\lambda} R\rangle \subseteq Z(R)$.
	\end{lemm}
	
	\begin{proof}
	Since \(d\) lies in the center of \(\operatorname{CDer}(R)\), it commutes with every element of \(\operatorname{CDer}(R)\).
	In particular, for any \(x \in R\), we have \([d_{\lambda} \operatorname{ad}_x] = 0\).
	This implies
	$$0 = [d_{\lambda} \operatorname{ad}_x]_{\lambda+\mu} y 
	= d_{\lambda}([x_{\mu} y]) - [x_{\mu} (d_{\lambda} y)] 
	= [(d_{\lambda} x)_{\lambda+\mu} y],\quad \forall y\in R$$
	Hence, for every \(x \in R\), the element \(d_{\lambda} x\) belongs to the center \(Z(R)\) of \(R\),
	and consequently $\langle d_{\lambda} R \rangle \subseteq Z(R)$.
	\end{proof}
	
	\begin{theo}\label{thm3.14}
	Let \(R\) be a Lie conformal algebra with trivial center, and let \(\circ_{\lambda}\) be a PLCA structure on \(R\). 
	Then there exists a \(\mathbb{C}[\partial]\)-linear map \(T: R \to \operatorname{CDer}(R)\), \(x \mapsto T(x)\), 
	such that \(x \circ_{\lambda} y = T(x)_{\lambda} y\) for all \(x, y \in R\). 
	Moreover, the map \(T\) satisfies the following identity for all \(x, y \in R\):
	\begin{equation}\label{eq3.13}
		\bigl[ T(x)_{\lambda} T(y) \bigr] = T\bigl( [T(x)_{\lambda} y] + [x_{\lambda} T(y)] + [x_{\lambda} y] \bigr).
	\end{equation}
	Here we identify \(R\) with its image under the adjoint representation \(\operatorname{ad}: R \hookrightarrow \operatorname{CDer}(R)\).
	\end{theo}
	
	\begin{proof}
		For any \(x \in R\), the left multiplication \(L(x) \in \operatorname{CDer}(R)\) defined by \(L(x)_{\lambda} y = x \circ_{\lambda} y\) is a conformal derivation. Hence there exists a map \(T: R \rightarrow \operatorname{CDer}(R)\), \(x \mapsto T(x)\), such that \(x \circ_{\lambda} y = T(x)_{\lambda} y\). For any \(x, y, z \in R\) and \(k_1, k_2 \in \mathbb{C}\), we have
		\begin{align*}
			T(\partial x)_{\lambda} y &= (\partial x) \circ_{\lambda} y = -\lambda (x \circ_{\lambda} y) = -\lambda T(x)_{\lambda} y = (\partial T(x))_{\lambda} y, \\
			T(k_1 x + k_2 y)_{\lambda} z &= (k_1 x + k_2 y) \circ_{\lambda} z = k_1 (x \circ_{\lambda} z) + k_2 (y \circ_{\lambda} z) = k_1 T(x)_{\lambda} z + k_2 T(y)_{\lambda} z.
		\end{align*}
		Thus \(T\) is a \(\mathbb{C}[\partial]\)-linear map. We may identify \(R\) with the subalgebra ${\rm CInn}(R) \subseteq \operatorname{CDer}(R)$.
		Using the definition \(x \circ_{\lambda} y = T(x)_{\lambda} y\) and the conformal Jacobi identity in \(\operatorname{CDer}(R)\), we obtain 
	\begin{equation*}
		\big[ [T(x)_{\lambda}T(y)]_{\lambda+\mu} z \big] = \Big[T\bigl([T(x)_{\lambda} y]\bigr)_{\lambda+\mu} z\Big] + \Big[T\bigl([x_{\lambda}T(y)]\bigr)_{\lambda+\mu} z\Big] + \big[T\bigl([x_{\lambda} y]\bigr)_{\lambda+\mu} z\big],\quad \forall x,y,z\in R.
	\end{equation*}
		By Lemma \ref{lem3.13}, the center of \(\operatorname{CDer}(R)\) is trivial when \(Z(R)=0\); hence the preceding equality of operators yields the identity \eqref{eq3.13}.
		
		Conversely, if a \(\mathbb{C}[\partial]\)-linear map \(T: R \rightarrow \operatorname{CDer}(R)\) satisfies \eqref{eq3.13}, then defining \(x \circ_{\lambda} y := T(x)_{\lambda} y\) endows \(R\) with a PLCA structure.
	\end{proof}
	
	\begin{coro}\label{cor3.15}
		Let \(R\) be a complete Lie conformal algebra. Then every PLCA structure on \(R\) is induced by a Rota--Baxter operator of weight $1$ on \(R\).
	\end{coro}	
	
	\begin{proof}
		Since \(R\) is complete, we have \(\operatorname{CDer}(R) = \operatorname{CInn}(R)\). Consequently, 
		\(R \cong \operatorname{CInn}(R) \cong \operatorname{CDer}(R)\). 
		By Theorem~\ref{thm3.14}, there exists a \(\mathbb{C}[\partial]\)-linear map 
		\(T: R \rightarrow \operatorname{CDer}(R)\) such that \(x \circ_{\lambda} y = T(x)_{\lambda} y\), and 
		\(T\) satisfies identity \eqref{eq3.13}. 
		Since $\operatorname{CDer}(R)=\operatorname{CInn}(R)\cong R$, the map $T$ may be regarded as a linear endomorphism of $R$.
		Under the identification \(\operatorname{ad}: R \xrightarrow{\sim} \operatorname{CInn}(R)\), we have 
		\(T(x)_{\lambda} y = (\operatorname{ad}_{T(x)})_{\lambda} y = [T(x)_{\lambda} y]\). 
		Substituting this expression into \eqref{eq3.13} yields precisely the Rota--Baxter relation for $T$ on $R$.
		Hence $T$ is a Rota--Baxter operator of weight 1 on $R$, and the given PLCA structure coincides with the one induced by $T$.
	\end{proof}
	
	\begin{theo}
		Every PLCA structure on the Virasoro conformal algebra $\rm Vir$ is trivial.
	\end{theo}
	
	\begin{proof}
		By \cite[Lemma 6.1]{AK}, every conformal derivation of the Virasoro conformal algebra $\mathrm{Vir}$ is inner. 
		Therefore, for any PLCA structure on $\mathrm{Vir}$, the $\lambda$-multiplication $\circ_\lambda$ must be of the form  
		\[
		L \circ_{\lambda} L = (\partial + 2\lambda) g(\lambda) L, ~ \text{for some } g(\lambda) \in \mathbb{C}[\lambda].
		\]
		The compatibility condition \eqref{eq3.2} gives the functional equation
		\[
		g(\lambda)g(\mu) = g(\lambda+\mu) \bigl(g(\lambda) + g(\mu) + 1 \bigr).
		\]
		Setting $\mu = 0$ yields $g(\lambda)(g(\lambda) + 1) = 0$, whose only solutions are $g(\lambda) = 0$ or $g(\lambda) = -1$. 
		These correspond to the two trivial PLCA structures:
		\[
		L \circ_\lambda L = 0, \qquad 
		L \circ_\lambda L = -(\partial + 2\lambda)L = -[L_\lambda L].
		\]
		This completes the proof.
	\end{proof}

\section{PLCA structures on Lie conformal algebras of Block-type}

In this section, we present a complete classification of Rota--Baxter operators of weight~$1$ on the Block-type Lie conformal algebras $\B(q)$. Under the assumption that $q \in \mathbb{C}$, $q \neq 0$, and $q \notin \mathbb{Z}_{-}$, every conformal derivation of $\B(q)$ is inner (see, e.g.,~\cite{SYY}). Based on this result, we further determine all non-isomorphic post-Lie conformal algebra structures on $\B(q)$.

\begin{defi}
	Let $q \in \mathbb{C}$. The Block-type Lie conformal algebra $\B(q)$
	is the free $\mathbb{C}[\partial]$-module with basis
	$\{L_i \mid i \in \mathbb{Z}_+\}$, whose $\lambda$-brackets are defined by
	\begin{equation}\label{eq4.1}
		[L_i{}_\lambda L_j]
		= \bigl((i+q)\partial + (i+j+2q)\lambda \bigr) L_{i+j},
		\qquad i, j \in \mathbb{Z}_+ .
	\end{equation}
\end{defi}

The Lie conformal algebra $\B(q)$ admits a natural
$\mathbb{Z}_+$-grading
\[
\B(q) = \bigoplus_{n \geq 0} \B_n,
\qquad
\B_n = \mathbb{C}[\partial] L_n .
\]
For each $i \in \mathbb{Z}_+$, we define the \emph{$i$-th truncated algebra} by
\[
M_i := \B(q) / I^{> i},
\qquad
I^{> i} := \bigoplus_{k > i} \B_k .
\]

Rota--Baxter operators of weight 1 on $\B(q)$ will be constructed inductively by
considering their restrictions to the truncated algebras $M_i$.

\begin{lemm}\label{lem4.2}
	Any Rota-Baxter operator of weight 1 on the Virasoro conformal algebra $\rm Vir$ must be trivial.
\end{lemm}
\begin{proof}
Let \(T\) be a Rota--Baxter operator of weight \(1\) on \(\mathrm{Vir}\).  
Write \(T(L) = f(\partial)L\) for some \(f(\partial) \in \mathbb{C}[\partial]\).  
By \eqref{eq2.2}, we have
\[
f(-\lambda) f(\partial + \lambda) 
= \bigl(f(-\lambda) + f(\partial + \lambda) + 1\bigr) f(\partial).
\]
Setting \(\lambda = 0\) in the above equality and simplifying yields
\[
\bigl(f(\partial) + 1\bigr) f(\partial) = 0.
\]
Hence \(f(\partial) = 0\) or \(f(\partial) = -1\), both of which correspond to trivial Rota--Baxter operators of weight 1.
\end{proof}
\begin{lemm}\label{lem4.3}
	Up to isomorphism, $M_0$ is the free Lie  conformal algebra of rank 1 with the $\lambda$-bracket
	\[
	[L_{0\,\lambda} L_0] = q(\partial + 2\lambda)L_0.
	\]
	Then every Rota--Baxter operator of weight 1 on $M_0$ is trivial.
\end{lemm}

\begin{proof}
	Since $q \neq 0$, the element $\widetilde L_0 := \frac{1}{q}L_0$ satisfies
	\[
	[\widetilde L_{0\,\lambda} \widetilde L_0] 
	= \frac{1}{q^2}[L_{0\,\lambda} L_0] 
	= \frac{1}{q^2} q(\partial+2\lambda)L_0 
	= (\partial+2\lambda)\widetilde L_0.
	\]
	Hence the $\mathbb{C}[\partial]$-linear map sending $\widetilde L_0$ to the canonical generator $L$ of $\mathrm{Vir}$ is an isomorphism of Lie conformal algebras, and so $M_0 \cong \mathrm{Vir}$.
	By Lemma~\ref{lem4.2}, all Rota--Baxter operators on $\mathrm{Vir}$ are trivial.
	Consequently, every Rota--Baxter operator on $M_0$ is trivial as well.
\end{proof}	

	\begin{lemm}\label{lem4.4}
	Up to isomorphism, the Lie conformal algebra
	$M_1 = \mathbb{C}[\partial] L_0 \oplus \mathbb{C}[\partial] L_1$
	is determined by the following $\lambda$-brackets:
	\begin{align*}
		[L_{1\lambda} L_1] = 0, \quad
		[L_{0\lambda} L_0] = q(\partial + 2\lambda) L_0, \quad
		[L_{0\lambda} L_1] = (q\partial + (1+2q)\lambda) L_1,\quad 
		[L_{1\lambda} L_0]=\bigl((1+q)\partial+(1+2q)\lambda\bigr) L_1 .
	\end{align*}
	Let $T$ be a Rota--Baxter operator of weight 1 on $M_1$. Then $T(L_1) = \alpha L_1,$
	where $\alpha = 0$ or $\alpha = -1$.
	\end{lemm}
	
	\begin{proof}	
By Lemmas~\ref{lem2.16} and~\ref{lem4.3}, we may restrict our attention to
Rota--Baxter operators $T$ on $M_1$ of the form
$$T(L_0) = f(\partial) L_1, \qquad
T(L_1) = g_1(\partial) L_0 + g_2(\partial) L_1,$$
where $f(\partial), g_1(\partial), g_2(\partial) \in \mathbb{C}[\partial]$.
Taking $x = y = L_0$ in \eqref{eq2.2}, we obtain

		\begin{align}
			&\Big( f(-\lambda)\big((1+q)\partial + (1+2q)\lambda\big) + f(\partial+\lambda)\big(q\partial + (1+2q)\lambda\big) \Big) g_1(\partial) = 0, \label{eq4.2} \\
			&\Big( f(-\lambda)\big((1+q)\partial + (1+2q)\lambda\big) + f(\partial+\lambda)\big(q\partial + (1+2q)\lambda\big) \Big) g_2(\partial) + f(\partial) q(\partial+2\lambda) = 0. \label{eq4.3}
		\end{align}
		
		 If $g_1(\partial) \neq 0$, in \eqref{eq4.2}, we deduce that
		\[
		f(-\lambda)\Big((1+q)\partial + (1+2q)\lambda\Big) + f(\partial+\lambda)\Big(q\partial + (1+2q)\lambda\Big) = 0,
		\] 
		which together with \eqref{eq4.3} implies $f(\partial) = 0$, and hence $T(L_0) = 0$.
		
		Now set $x = L_0$, and $y = L_1$ in \eqref{eq2.2} to obtain
		\[
		\left( g_2(\partial+\lambda) + 1 \right) T(L_1) = 0.
		\]
		Since $T(L_1) \neq 0$ (as $g_1(\partial) \neq 0$), it follows that $g_2(\partial+\lambda) + 1 = 0$, and thus $g_2(\partial) = -1$.
		
		Next, set $x = y = L_1$ in \eqref{eq2.2}
		\begin{align}\label{eq:3}
			q(\partial+2\lambda)g_1(\partial+\lambda) g_1(-\lambda)  = \Big( g_1(\lambda)\big(q\partial + (1+2q)\lambda\big) + g_1(\partial+\lambda)\big((1+q)\partial + (1+2q)\lambda\big) \Big) g_1(\partial). 
		\end{align}	
		If $\deg g_1(\partial) \geq 1$, then let $\partial_0$ be a root of $g_1(\partial)$. Substituting $\partial = \partial_0$ into \eqref{eq:3} yields $q(\partial_0+2\lambda)g_1(\partial_0+\lambda) g_1(-\lambda)=0$,
		which is impossible. Therefore, $g_1(\partial)$ must be constant. Let $g_1(\partial) = c$. Then equation \eqref{eq:3} becomes 
		$$c^2 q(\partial+2\lambda) = c^2 (1+2q)(\partial+2\lambda),$$
		which implies $c^2 (1+q)(\partial+2\lambda) = 0.$ Since $q \ne-1$, we must have $c = 0$, contradicting $g_1(\partial) \neq 0$. Thus, we obtain 
		$$g_1(\partial) = 0.$$
		Then $T(L_1) = g_2(\partial) L_1$. Setting $x = L_0$ and $y = L_1$ in \eqref{eq2.2} gives
		$$T\left( [L_{0\,\lambda} T(L_1)] + [L_{0\,\lambda} L_1] \right) = 0.$$
		Note that $[L_{0\,\lambda} T(L_1)] = g_2(\partial+\lambda) [L_{0\,\lambda} L_1] = g_2(\partial+\lambda) \big(q\partial + (1+2q)\lambda\big) L_1,$
		and $[L_{0\,\lambda} L_1] = \big(q\partial + (1+2q)\lambda\big) L_1.$ Thus, the sum is $\left(g_2(\partial+\lambda) + 1\right) \big(q\partial + (1+2q)\lambda\big) L_1$. Applying $T$ yields $$\left(g_2(\partial+\lambda) + 1\right) \big(q\partial + (1+2q)\lambda\big) g_2(\partial) L_1 = 0,$$
		so we obtain $\left(g_2(\partial+\lambda) + 1\right) g_2(\partial) = 0.$
		Therefore, either $g_2(\partial) = 0$ or $g_2(\partial) = -1$, which implies $T(L_1) = 0$ or $T(L_1) = -L_1$.
	\end{proof}

	\begin{theo}\label{them4.5}
		Let $n \geq 1$, and let $T$ be a Rota--Baxter operator of weight $1$ on $M_n$.
		Then for each integer $i$ with $1 \leq i \leq n$, one has $T(L_i) = \alpha L_i,$
		where $\alpha \in \{0,-1\}$.
	\end{theo}
	
	\begin{proof}
		We proceed by induction on $n$. By Lemma~\ref{lem4.4}, the statement holds for
		$n=1$. Assume that it holds for $M_k$ with $1 \leq k \leq n$, where $n \geq 1$,
		and consider the truncated algebra $M_{n+1}$.
		By Lemmas~\ref{lem2.16} and~\ref{lem4.3}, we may assume that
		$T(L_0) \in \bigoplus_{i=1}^{n+1} \B_i.$
		By the induction hypothesis, we may further write
		$$T(L_i) = \alpha L_i + f_i(\partial) L_{n+1},\quad
		f_i(\partial) \in \mathbb{C}[\partial], \quad 1 \leq i \leq n,\quad \alpha\in\{0,-1\}.$$
		Let $
		k_1 = \Big\lfloor \frac{n+1}{2} \Big\rfloor,~
		k_2 = \Big\lfloor \frac{n+2}{2} \Big\rfloor.$
		Then $1 \leq k_1 \leq k_2 \leq n$ and $k_1 + k_2 = n+1$.
		Substituting $x = L_{k_1}$ and $y = L_{k_2}$ into \eqref{eq2.2}, we obtain
		\[
		[T(L_{k_1})_\lambda T(L_{k_2})]
		=
		T\bigl(
		[T(L_{k_1})_\lambda L_{k_2}]
		+
		[L_{k_1\,\lambda} T(L_{k_2})]
		+
		[L_{k_1\,\lambda} L_{k_2}]
		\bigr).
		\]
		Since $[L_{i\,\lambda} L_{n+1}] = 0$ for all $i \geq 1$, this identity reduces to
		\[
		\alpha^2 [L_{k_1\,\lambda} L_{k_2}]
		=
		T\bigl((2\alpha+1)[L_{k_1\,\lambda} L_{k_2}]\bigr).
		\]
		Consequently,
		\[
		\bigl((k_1+q)\partial + (n+1+2q)\lambda\bigr)
		\bigl((2\alpha+1)T(L_{n+1}) - \alpha^2 L_{n+1}\bigr) = 0.
		\]
		Since $(k_1+q)\partial + (n+1+2q)\lambda \neq 0$, it follows that
		\[
		T(L_{n+1}) = \frac{\alpha^2}{2\alpha+1} L_{n+1}.
		\]
		For $\alpha \in \{0,-1\}$, one has $\frac{\alpha^2}{2\alpha+1} = \alpha$, and hence
		\[
		T(L_{n+1}) = \alpha L_{n+1}.
		\]
		
		It remains to show that $f_i(\partial) = 0$ for all $1 \leq i \leq n$.
		Suppose to the contrary that there exists $i$ such that $f_i(\partial) \neq 0$,
		and let $n_0$ be the largest positive integer with this property.
		Substituting $x = L_0$ and $y = L_{n_0}$ into \eqref{eq2.2}, we obtain
		\[
		[T(L_0)_\lambda T(L_{n_0})]
		=
		T\bigl(
		[T(L_0)_\lambda L_{n_0}]
		+
		[L_{0\,\lambda} T(L_{n_0})]
		+
		[L_{0\,\lambda} L_{n_0}]
		\bigr).
		\]
		Since
		\[
		[T(L_0)_\lambda T(L_{n_0})]
		=
		[T(L_0)_\lambda \alpha L_{n_0}]
		=
		\alpha [T(L_0)_\lambda L_{n_0}]
		=
		T([T(L_0)_\lambda L_{n_0}]),
		\]
		the above identity reduces to
		\[
		T\bigl((\alpha+1)[L_{0\,\lambda} L_{n_0}]\bigr)
		+
		f_{n_0}(\partial+\lambda)\, T([L_{0\,\lambda} L_{n+1}]) = 0.
		\]
		If $\alpha = -1$, then $f_{n_0}(\partial+\lambda)\, T([L_{0\,\lambda} L_{n+1}]) = 0,$
		which implies $f_{n_0}(\partial) = 0$ since
		$T([L_{0\,\lambda} L_{n+1}]) \neq 0$.
		If $\alpha = 0$, then $T([L_{0\,\lambda} L_{n_0}]) = 0$, which implies
		$T(L_{n_0}) = 0$ and hence $f_{n_0}(\partial) = 0$.
		Both cases contradict the choice of $n_0$.
		
		Therefore, $f_i(\partial) = 0$ for all $1 \leq i \leq n$, and the induction is
		complete.
	\end{proof}

	\begin{prop}\label{prop4.6}
		Let $T$ be a Rota--Baxter operator on $\B(q)$. Then for every $i \geq 1$, one has $T(L_i) = \alpha L_i,$	where $\alpha \in \{0,-1\}$.
	\end{prop}
	
	\begin{proof}
		Suppose, to the contrary, that there exists an integer $n \geq 1$ such that
		$T(L_n) \neq \alpha L_n$ for both $\alpha = 0$ and $\alpha = -1$.
		Write $T(L_n) = \sum_{i=0}^{m} f_i(\partial) L_i$
		and set $N = \max\{m,n\}$.
		By Remark~\ref{re2.15}, the restriction $T|_{M_N}$ is a Rota--Baxter operator of
		weight~$1$ on $M_N$.
		However, Theorem~\ref{them4.5} implies that $T|_{M_N}(L_n) = \alpha L_n$
		for some $\alpha \in \{0,-1\}$, which contradicts the choice of $n$.
		Therefore, the proposition holds.
	\end{proof}

	\begin{prop}\label{prop4.7}
		Let $T$ be a Rota--Baxter operator on $\mathfrak{B}(q)$. Then $T(L_0) = \beta L_0$, where $\beta = 0$ or $-1$.
	\end{prop}
	
	\begin{proof}
		Write $T(L_0)=\sum_{i=0}^{m} f_i(\partial)L_i,~f_m(\partial)\neq 0$.
		Assume that $m>0$. Then
		$$[T(L_0)_\lambda T(L_0)]=f_m(-\lambda)f_m(\partial+\lambda)[L_{m\,\lambda}L_m] + L,$$
		where $L\in M_{2m-1}$.
		On the other hand, by Proposition~\ref{prop4.6} we have
		$$T\bigl([T(L_0)_\lambda L_0]+[L_{0\,\lambda}T(L_0)]+[L_{0\,\lambda}L_0]\bigr)
		\in T(M_m)\subseteq M_m.$$
		Consequently,
		$$f_m(-\lambda)f_m(\partial+\lambda)[L_{m\,\lambda}L_m]
		\in M_{2m-1}+M_m \subseteq M_{2m-1}.$$
		This is impossible, since $[L_{m\,\lambda}L_m]\notin M_{2m-1}$.
		Therefore $m=0$, and hence $T(L_0)=f_0(\partial)L_0$.
		Applying Lemma~\ref{lem4.3} to the subalgebra $M_0$, we obtain
		$$T(L_0)=T|_{M_0}(L_0)=\beta L_0,$$
		where $\beta = 0$ or $-1$.
	\end{proof}
	
	\begin{theo}\label{them4.8}
		Let $T$ be a Rota--Baxter operator on $\mathfrak{B}(q)$. Then $T$ must be one of the following forms:
		\begin{itemize}
			\item[(i)] $T$ is trivial, i.e., $T = 0$ or $T = -\mathrm{id}$;
			\item[(ii)] $T(L_0) = 0$, $T(L_i) = -L_i$ for all $i \geqslant 1$;
			\item[(iii)] $T(L_0) = -L_0$, $T(L_i) = 0$ for all $i \geqslant 1$.
		\end{itemize}
	\end{theo}
	
	\begin{proof}
		By Propositions~\ref{prop4.6} and~\ref{prop4.7}, we have $T(L_0)=\beta L_0$ and$T(L_i)=\alpha L_i$ for all $i\in\mathbb{N}$, where $\alpha,\beta\in\{0,-1\}$. Conversely, one can verify directly that any map $T$ satisfying such conditions is indeed a Rota--Baxter operator on $\mathfrak{B}(q)$.
	\end{proof}
	
	\begin{theo}
		The PLCA structures on the Lie conformal algebra $\mathfrak{B}(q)$ are precisely one of the following:
		\begin{enumerate}[label=(\roman*)]
			\item $x \circ_\lambda y = 0$ for all $x,y\in\mathfrak{B}(q)$;
			\item $x \circ_\lambda y = -[x_\lambda y]$ for all $x,y\in\mathfrak{B}(q)$;
			\item $L_0 \circ_\lambda x = 0$ and $L_i \circ_\lambda x = -[L_{i\lambda}x]$ for all $i\in\mathbb{N}$ and $x\in\mathfrak{B}(q)$;
			\item $L_0 \circ_\lambda x = -[L_{0\lambda}x]$ and $L_i \circ_\lambda x = 0$ for all $i\in\mathbb{N}$ and $x\in\mathfrak{B}(q)$.
		\end{enumerate}
	\end{theo}
	
	\begin{proof}
		Since $q\neq0$ and $q\notin\mathbb{Z_-}$, the Lie conformal algebra $\mathfrak{B}(q)$ is complete.
		By Corollary~\ref{cor3.15}, any PLCA structure on $\mathfrak{B}(q)$ is induced by a Rota--Baxter operator $T$ through
		\[
		x\circ_\lambda y=[T(x)_\lambda y],\qquad x,y\in\mathfrak{B}(q).
		\]
		All Rota--Baxter operators on $\mathfrak{B}(q)$ are classified in Theorem~\ref{them4.8}.
		Substituting these operators into the above formula yields exactly the four PLCA structures listed above.
	\end{proof}

	\section{PLCA structures on the Lie conformal algebras $W(b)$}
	In this section, we will classify all PLCA structures on the Lie conformal algebras $W(b)$.
	
	For $b\in\mathbb{C}$, the Lie conformal algebra ${W}(b)$ is a free $\mathbb{C}[\partial]$-module generated by $L$ and $H$, satisfying the following $\lambda$-brackets:
	\begin{equation} \label{eq5.1}
		[L_\lambda L]=(\partial+2\lambda)L,\quad[L_\lambda H]=\big(\partial+(1-b)\lambda\big)H, \quad[H_\lambda L]=\big(-b\partial+(1-b)\lambda\big)H, \quad[H_\lambda H]=0.  
	\end{equation}
	This algebra was introduced in \cite{XY} and later studied in \cite{LY, WY}. When $b=0$, it recovers the Heisenberg–Virasoro conformal algebra, which first appeared in \cite{SY} and whose cohomology with trivial coefficients was computed in \cite{YW1}. For $b=-1$, it becomes the
	W(2,2) Lie conformal algebra, whose structural properties were investigated in \cite{YW2}.

	\subsection{Rota-Baxter operators of weight 1 on $W(b)$}

	Notice that $W(b)$ has a nontrivial abelian conformal ideal $\C[\p]H$, and the quotient $W(b)/{\C[\p]H}$ is isomorphic to the Virasoro conformal algebra $\rm Vir$.

	\begin{prop}\label{prop5.1}
		We have
		$${\rm CDer}(W(b))=\left\{
		\begin{aligned}
			&{\rm CInn}(W(b)),\quad\quad \quad\text{if}~b\ne0,\\
			&{\rm CInn}(W(0))\oplus\C D_{\l},~~ \text{if}~ b=0,
		\end{aligned}
		\right.$$
		where $D_{\l}$ is a conformal outer derivation of $W(0)$, defined by $$D_{\l}(L)=H,~D_{\l}(H)=0.$$
	\end{prop}
	
	\begin{proof}
		Let $d_\lambda$ be a conformal derivation of $W(b)$. Write
		\begin{align*}
			d_{\lambda}(L)=u_{1}(\partial,\lambda)L+u_{2}(\partial,\lambda)H,\quad
			d_{\lambda}(H)=v_{1}(\partial,\lambda)L+v_{2}(\partial,\lambda)H,
		\end{align*}
		where $u_i(\partial,\lambda),v_i(\partial,\lambda)\in\mathbb{C}[\partial,\lambda]$.
		Taking $x=L$ and $y=H$ in \eqref{eq2.1}, we obtain
		\begin{align}
			\label{eq5.2}
			(\partial+\lambda+(1-b)\mu)v_{1}(\partial,\lambda)
			&=(\partial+2\lambda)v_{1}(\partial+\mu,\lambda),\\
			\label{eq5.3}
			(\partial+\lambda+(1-b)\mu)v_{2}(\partial,\lambda)
			&=(\partial+(1-b)(\lambda+\mu))u_{1}(-\lambda-\mu,\lambda)
			+(\partial+(1-b)\mu)v_{2}(\partial+\mu,\lambda).
		\end{align}
		Setting $\mu=0$ in \eqref{eq5.2} yields $\lambda v_{1}(\partial,\lambda)=0$, hence
		\begin{align}
			\label{eq5.4}
			v_{1}(\partial,\lambda)=0.
		\end{align}
		
		Next, taking $x=y=L$ in \eqref{eq2.1}, we have
		\begin{align}
			\label{eq5.5}
			(\partial+\lambda+2\mu)u_{1}(\partial,\lambda)
			&=(\partial+2\lambda+2\mu)u_{1}(-\lambda-\mu,\lambda)
			+(\partial+2\mu)u_{1}(\partial+\mu,\lambda),\\
			\label{eq5.6}
			(\partial+\lambda+2\mu)u_{2}(\partial,\lambda)
			&=(-b\partial+(1-b)(\lambda+\mu))u_{2}(-\lambda-\mu,\lambda)
			+(\partial+(1-b)\mu)u_{2}(\partial+\mu,\lambda).
		\end{align}
		Substituting $\partial=-\lambda-\mu$ into \eqref{eq5.5}, we obtain
		\begin{align}
			\label{eq5.7}
			(\lambda-\mu)u_{1}(-\lambda,\lambda)=\lambda u_{1}(-\lambda-\mu,\lambda).
		\end{align}
		Since $\lambda$ and $\lambda-\mu$ are relatively prime, it follows that
		$\lambda\mid u_{1}(-\lambda,\lambda)$. Write
		$u_{1}(-\lambda,\lambda)=\lambda p_{1}(\lambda)$ with
		$p_{1}(\lambda)\in\mathbb{C}[\lambda]$. Then \eqref{eq5.7} implies
		\begin{align}
			\label{eq5.8}
			u_{1}(\partial,\lambda)=(\partial+2\lambda)p_{1}(\lambda).
		\end{align}
		Setting $\mu=0$ in \eqref{eq5.3} and using \eqref{eq5.8}, we obtain
		\begin{align}
			\label{eq5.9}
			v_{2}(\partial,\lambda)=(\partial+(1-b)\lambda)p_{1}(\lambda).
		\end{align}
		Taking $\partial=-\lambda-\mu$ in \eqref{eq5.6}, we have
		\begin{align}
			\label{eq5.10}
			\lambda u_{2}(-\lambda-\mu,\lambda)=(\lambda+b\mu)u_{2}(-\lambda,\lambda).
		\end{align}
		Assume first that $b\neq0$. Then $\lambda$ and $\lambda+b\mu$ are relatively prime, and
		\eqref{eq5.10} implies $\lambda\mid u_{2}(-\lambda,\lambda)$. Hence
		\begin{align}
			\label{eq5.11}
			u_{2}(-\lambda,\lambda)=\lambda p_{2}(\lambda),
		\end{align}
		for some $p_{2}(\lambda)\in\mathbb{C}[\lambda]$. Setting $\mu=0$ in
		\eqref{eq5.6} and using \eqref{eq5.11}, we obtain
		\begin{align}
			\label{eq5.12}
			u_{2}(\partial,\lambda)=\big(-b\partial+(1-b)\lambda\big)p_{2}(\lambda).
		\end{align}
		Let $a_{1}=p_{1}(-\partial)L+p_{2}(-\partial)H$. By
		\eqref{eq5.4}, \eqref{eq5.8}, \eqref{eq5.9}, and \eqref{eq5.12}, we have 
		\[
		d_{\lambda}(L)=({\rm ad}\,a_{1})_{\lambda}(L),\qquad
		d_{\lambda}(H)=({\rm ad}\,a_{1})_{\lambda}(H).
		\]
		Therefore $d_{\lambda}$ is an inner conformal derivation.
		
		\medskip
		
		Now assume that $b=0$. Setting $\mu=0$ in \eqref{eq5.6}, we obtain
		$u_{2}(\partial,\lambda)=u_{2}(-\lambda,\lambda)$. Hence
		\begin{align}
			\label{eq5.13}
			u_{2}(\partial,\lambda)=\lambda p(\lambda)+c,
		\end{align}
		for some $p(\lambda)\in\mathbb{C}[\lambda]$ and $c\in\mathbb{C}$.
		Let $a_{2}=p_{1}(-\partial)L+p(-\partial)H$. By
		\eqref{eq5.4}, \eqref{eq5.8}, \eqref{eq5.9}, and \eqref{eq5.13}, we have 
		\[
		d_{\lambda}(L)=({\rm ad}\,a_{2})_{\lambda}(L)+cD_{\lambda}(L),\qquad
		d_{\lambda}(H)=({\rm ad}\,a_{2})_{\lambda}(H)+cD_{\lambda}(H).
		\]
		Therefore $d_{\lambda}=({\rm ad}\,a_{2})_{\lambda}+cD_{\lambda}$. This completes the proof.
	\end{proof}
	
\begin{theo}\label{them5.2}
	Let $T$ be a nontrivial Rota--Baxter operator of weight $1$ on $W(b)$, where $b \neq 0$. Recall that the operator \(T'\), defined by \(T' = -T - \mathrm{id}\), is also a Rota--Baxter operator of weight $1$. Then \(T\) and \(T'\) take the following forms:
	\begin{itemize}
		\item[(1)] When $b = 1$:
		\[
		T(L) = (a \p^{2} + c \p) H, \quad T(H) = -H, \quad
		T'(L) = -L - (a \p^{2} + c \p) H, \quad T'(H) = 0,~~ a,~c\in\mathbb{C};
		\]
		\item[(2)] When $b = 2$:
		\[
		T(L) = (a \p^{3} + c \p) H, \quad T(H) = -H, \quad
		T'(L) = -L - (a \p^{3} + c \p) H, \quad T'(H) = 0, ~~ a,~c\in\mathbb{C};
		\]
	
		\item[(3)] When $b \neq 1, 2$:
		\[
		T(L) = c \p H, \quad T(H) = -H, \quad
		T'(L) = -L - c \p H, \quad T'(H) = 0, ~~c\in\mathbb{C}.
		\]
		
	\end{itemize}  
\end{theo}
	\begin{proof}
		It is clear that $W(b)/\mathbb{C}[\partial]H \cong \mathrm{Vir}$.
		By Lemma~\ref{lem4.2}, a Rota--Baxter operator $T$ on $W(b)$ satisfies either
		$T(L)=f(\partial)H$ or $T(L)=-L+f(\partial)H$.
		Replacing $T$ by $-T-\mathrm{id}$ if necessary, we may assume
		\[
		T(L)=f(\partial)H,\qquad 
		T(H)=g_1(\partial)L+g_2(\partial)H.
		\]
		
		Taking $x=y=H$ in \eqref{eq2.2} and comparing the coefficients of $L$, we obtain
		\begin{align}\label{eq5.14}
			(\partial+2\lambda)g_1(-\lambda)g_1(\partial+\lambda)
			=\Big(\big(\partial+(1-b)\lambda\big)g_1(-\lambda)
			+\big(-b\partial+(1-b)\lambda\big)g_1(\partial+\lambda)\Big)g_1(\partial).
		\end{align}
		Setting $\lambda=0$ in \eqref{eq5.14} gives $b\partial g_1^2(\partial)=0$.
		Since $b\neq0$, it follows that $g_1(\partial)=0$.
		
		Next, setting $x=L$, $y=H$ in \eqref{eq2.2}, we obtain
		\[
		(\partial+(1-b)\lambda)\bigl(g_2(\partial+\lambda)+1\bigr)g_2(\partial)=0,
		\]
		which implies $g_2(\partial)=0$ or $g_2(\partial)=-1$.
		
		Now take $x=y=L$ in \eqref{eq2.2}. Then
		\begin{align*}
			\big((-b\partial+(1-b)\lambda)f(-\lambda)
			+(\partial+(1-b)\lambda)f(\partial+\lambda)\big)g_2(\partial)
			+(\partial+2\lambda)f(\partial)=0.
		\end{align*}
		If $g_2(\partial)=0$, then $f(\partial)=0$, and hence $T\equiv0$.
		Thus we assume $g_2(\partial)=-1$, in this case the above identity reduces to
		\begin{align}\label{eq5.15}
			(-b\partial+(1-b)\lambda)f(-\lambda)
			+(\partial+(1-b)\lambda)f(\partial+\lambda)
			=(\partial+2\lambda)f(\partial).
		\end{align}
		
		Setting $\lambda=0$ in \eqref{eq5.15} yields $b\partial f(0)=0$, hence $f(0)=0$.
		Therefore, $f(\partial)=\partial f_0(\partial)$ for some $f_0(\partial)\in\mathbb{C}[\partial]$.
		Substituting this into \eqref{eq5.15}, we obtain
		\begin{align*}
			(-b\partial+(1-b)\lambda)(-\lambda)f_0(-\lambda)
			+(\partial+(1-b)\lambda)(\partial+\lambda)f_0(\partial+\lambda)
			=(\partial+2\lambda)\partial f_0(\partial).
		\end{align*}
		After the change of variables $(\partial+\lambda,-\lambda)\to(\partial,\lambda)$,
		this equation becomes
		\begin{align}\label{eq5.16}
			\partial(\partial+b\lambda)\bigl(f_0(\partial+\lambda)-f_0(\partial)\bigr)
			=
			\lambda(b\partial+\lambda)\bigl(f_0(\partial+\lambda)-f_0(\lambda)\bigr).
		\end{align}
		It is clear that \eqref{eq5.16} is automatically satisfied when
		$\deg(f_0)=0$.
		Assume henceforth that $\deg(f_0)\ge 1$, and write $f_0(\lambda)=\sum_{i=0}^{n} a_i \lambda^i,~a_n \neq 0.$
		We compare the highest-degree terms in $\partial$ on both sides of
		\eqref{eq5.16}.
		On the left-hand side, the leading contribution comes from
		$$\partial(\partial+b\lambda)\Big((\partial+\lambda)^n-\p^{n}\Big),$$
		whose highest-degree term in $\partial$ is $n\l\p^{n+1}$.
		On the right-hand side, the leading contribution arises from
		\[
		\lambda(b\partial+\lambda)(\partial+\lambda)^n,
		\]
		whose highest-degree term in $\partial$ is $b\l\p^{n+1}$.
		In order for the degrees in $\partial$ on both sides of \eqref{eq5.16}
		to coincide, the coefficient of $\partial^{n+1}$ on both sides
		must be the same, which forces $n=b$.
		If $n>2$, then comparing the coefficients of $\partial^{\,n}$ in
		\eqref{eq5.16} yields
		\[
		\binom{n}{2} a_n \lambda^2 + a_{n-1}(n-1)\lambda = a_n \lambda^2,
		\]
		which is impossible since $\binom{n}{2}>1$ for $n>2$.
		Therefore, one must have $n=b=1$ or $n=b=2$.
		In the case of $b=2$, then \eqref{eq5.16} further implies $a_1=0$, and hence
		$f_0(\partial)=a_2\partial^2+a_0.$
		In the case of $b=1$, then \eqref{eq5.16} is identically satisfied, and
		$f_0(\partial)=a_1\partial+a_0.$
		Letting $a=a_n$ and $c=a_0$, we recover precisely the operator $T$
		described in Theorem~\ref{them5.2}.
		The remaining operator $T'$ is obtained by $T'=-T-\mathrm{id}$.
		
	\end{proof}
	
	\subsection{PLCA structures on $W(b)$}
		
	In this subsection, we classify all PLCA structures on $W(b)$.
	That is, we determine all $\l$-products such that $(W(b), \circ_\lambda, [\cdot_\lambda\cdot])$ is a post-Lie conformal algebra,
	where the $\lambda$-bracket is given by \eqref{eq5.1}.
	
	As observed in Section~3.1, if $(W(b),\circ_\lambda,[\cdot_\lambda\cdot])$ is a PLCA,
	then the left multiplication operator $L(x)_\lambda$ associated with $\circ_\lambda$
	defines a conformal derivation of $W(b)$.
	When $b\neq0$, Proposition~\ref{prop5.1} shows that every conformal derivation of $W(b)$
	is inner.
	Since the center of $W(b)$ is trivial, the Lie conformal algebra $W(b)$ is complete.
	Consequently, by Corollary~\ref{cor3.15}, every PLCA structure on $W(b)$ is induced by
	a Rota--Baxter operator $T$ of weight~$1$, and is therefore of the form
	\[
	x\circ_\lambda y=[T(x)_\lambda y],\qquad x,y\in W(b).
	\]
	
As a consequence of Theorem~\ref{them5.2}, which classifies all
Rota--Baxter operators of weight~$1$ on $W(b)$, we obtain a complete list
of nontrivial PLCA structures on $W(b)$ for $b\neq0$, as follows (where $a$ and $c$ are constants in $\mathbb{C}$)
	
	\begin{itemize}
		\item [(1)] For $b=1$, we have 
		\begin{align}\label{eq5.17}
			L\circ_{\l}L=-\p(a \l^{2}-c \l)H,~~
			L\circ_{\l}H=0,~~
			H\circ_{\l}L=-\p H,~~
			H\circ_{\l}H=0,
		\end{align}
		and
		\begin{align}\label{eq5.18}
			L\circ_{\l}L=-(\p+2\l)L+\p(a \l^{2}-c \l)H,~~
			L\circ_{\l}H=-\p H,~~
			H\circ_{\l}L=0,~~
			H\circ_{\l}H=0.
		\end{align}
		\item [(2)] For $b=2$, we have 
		\begin{align}\label{eq5.19}
			L\circ_{\l}L=(a \l^{3}+c \l)(2\p+\l)H,~~
			L\circ_{\l}H=0,~~
			H\circ_{\l}L=-(2\p+\l)H,~~
			H\circ_{\l}H=0,
		\end{align}
		and
		\begin{align}\label{eq5.20}
			L\circ_{\l}L=-(\p+2\l)L-(a \l^{3}+c \l)(2\p+\l)H,~~
			L\circ_{\l}H=-(\p-\l)H,~~
			H\circ_{\l}L=0,~~
			H\circ_{\l}H=0.
		\end{align}
		\item [(3)] For $b\ne0,1,2$, we have 
		\begin{align}\label{eq5.21}
			L\circ_{\l}L=-c \l(-b\p+(1-b)\l)H,~~
			L\circ_{\l}H=0,~~
			H\circ_{\l}L=(-b\p+(1-b)\l)H,~~
			H\circ_{\l}H=0,
		\end{align} 
		and
		\begin{align}\label{eq5.22}
			L\circ_{\l}L=-(\p+2\l)L+c \l(-b\p+(1-b)\l)H,~~
			L\circ_{\l}H=-(\p+(1-b)\l)H,~~
			H\circ_{\l}L=0,~~
			H\circ_{\l}H=0.
		\end{align}
	\end{itemize}

Note that the $\lambda$-brackets for $W(1)$ are given by
$$[L_{\lambda}L]=(\partial+2\lambda)L,\quad[L_{\lambda}H]=\partial H,\quad[H_{\lambda}L]=-\partial H,\quad[H_{\lambda}H]=0.$$
Let $N=\partial H$. Then
$$[L_{\lambda}L]=(\partial+2\lambda)L,\quad[L_{\lambda}N]=(\partial+\lambda)N,\quad[N_{\lambda}L]=\lambda N,\quad[N_{\lambda}N]=0.$$
Hence, the Lie conformal algebra generated by $L$ and $N$ is isomorphic to $W(0)$.
In particular, $W(0)$ can be embedded as a Lie conformal subalgebra of $W(1)$.
Recall that $W(0)$ admits a nontrivial outer conformal derivation $D_{\lambda}$,
which satisfies
$$D_{\lambda}(L) = N,\qquad D_{\lambda}(N) = 0.$$
Under the above embedding, this derivation is realized as an inner derivation of $W(1)$:
\[
D_{\lambda}x = [(-H)_{\lambda}x],\qquad x\in W(0).
\]
Therefore, every conformal derivation of $W(0)$ is induced by the adjoint action of
an element of $W(1)$, and we have
\[
\operatorname{CDer}(W(0)) \cong W(1).
\]

By Theorem~\ref{thm3.14}, there exists a $\mathbb{C}[\partial]$-linear map
$T\colon W(0)\to W(1)$ such that
\[
x\circ_{\lambda}y = L(x)_{\lambda}y = [T(x)_{\lambda}y],
\qquad x,y\in W(0),
\]
and $T$ satisfies the identity \eqref{eq3.13}.
Arguing as in the proof of Theorem~\ref{them5.2}, we may assume that
\begin{equation}\label{eq5.23}
	T(L)=f(\partial)H,\qquad
	T(N)=g_{1}(\partial)L+g_{2}(\partial)H.
\end{equation}
\begin{theo}\label{thm5.3}
	Let $T$ be the operator defined by \eqref{eq5.23}.
	Then $T$ must be of one of the following forms:
	\begin{enumerate}[label=(\roman*)]
		\item $T(L)=0$, \quad $T(N)=cL-N$;
		\item $T(L)=-L$, \quad $T(N)=-cL$;
		\item $T(L)=(a_{1}\partial+a_{0})N$, \quad $T(N)=-N$;
		\item $T(L)=-L-(a_{1}\partial+a_{0})N$, \quad $T(N)=0$,
	\end{enumerate}
	where $a_{0},a_{1},c\in\mathbb{C},~c\neq0$.
\end{theo}
\begin{proof} 
	Setting $x = y = L$ in \eqref{eq3.13} yields
	\begin{align}
		\label{eq5.24} &\big(f(\partial + \lambda) - f(-\lambda)\big)g_{1}(\partial) = 0, \\
		\label{eq5.25} &\big(f(\partial + \lambda) - f(-\lambda)\big)g_{2}(\partial) + (\partial + 2\lambda)f(\partial) = 0.
	\end{align}
	
	\noindent\textbf{Case 1.} $g_{1}(\partial) \neq 0$.
	
	From \eqref{eq5.24}, we deduce that $f(\partial) = 0$, and hence $T(L) = 0$. Setting $x = L$, $y = N$ in \eqref{eq3.13}, we obtain
	\[
	\left((\partial+\lambda) + g_{2}(\partial + \lambda)\right)T(N) = 0.
	\]
	Since $g_1(\partial) \neq 0$ implies $T(N) \neq 0$, it follows that $g_{2}(\partial) = -\partial$. Setting $x = y = N$ in \eqref{eq3.13}, we obtain
	\[
	(\partial + 2\lambda)g_{1}(-\lambda)g_{1}(\partial + \lambda) = \big((\partial + \lambda)g_{1}(-\lambda) + \lambda g_{1}(\partial + \lambda)\big)g_{1}(\partial).
	\]
	Rewriting this equation gives
	\begin{align}\label{eq5.26}
		(\partial + \lambda)g_1(-\lambda)\big(g_1(\partial + \lambda) - g_1(\partial)\big) = \lambda g_1(\partial + \lambda)\big(g_1(\partial) - g_1(\lambda)\big).
	\end{align}
	Write $g_1(\partial) = \sum_{i=0}^{n} a_i \partial^i$ with $a_{n} \neq 0$. Comparing the highest-degree terms in $\partial$ (i.e., $\partial^{2n}$) on both sides of \eqref{eq5.26}, we conclude that $g_{1}(\partial)$ must be a nonzero constant. Therefore, we obtain
	\begin{align*}
		T(L) = 0, \quad T(N) = cL - \partial H,
	\end{align*}
	where $c \neq 0$ is a constant. Replacing $\partial H$ with $N$, we obtain part (i). Replacing $T$ by $-T - \mathrm{id}$ yields part (ii).
	
	\noindent\textbf{Case 2.} $g_{1}(\partial) = 0$.
	
	In this case, we have
	\[
	T(L) = f(\partial)H, \quad T(N) = g_{2}(\partial)H.
	\]
	If $g_{2}(\partial) = 0$, then \eqref{eq5.25} implies $f(\partial) = 0$, and hence $T = 0$. If $g_{2}(\partial) \neq 0$, then setting $x = L$, $y = N$ in \eqref{eq3.13} yields
	\begin{align*}
		\big(g_{2}(\partial + \lambda) + (\partial + \lambda)\big)g_{2}(\partial) = 0,
	\end{align*}
	which implies $g_{2}(\partial) = -\partial$ since $g_{2}(\partial) \neq 0$. Substituting this into \eqref{eq5.25}, we obtain
	\begin{align}\label{eq5.27}
		-\partial f(\partial + \lambda) + \partial f(-\lambda) + (\partial + 2\lambda)f(\partial) = 0.
	\end{align}
	Setting $\partial = 0$ in \eqref{eq5.27} yields $f(0) = 0$. Let $f(\partial) = \partial h(\partial)$; then \eqref{eq5.27} becomes
	\begin{align}\label{eq5.28}
		(\partial + 2\lambda)h(\partial) = (\partial + \lambda)h(\partial + \lambda) + \lambda h(-\lambda).
	\end{align}
	Suppose $\deg(h) > 1$ and write $h(\partial) = \sum_{i=0}^{n} a_{i} \partial^{i}$ with $n > 1$ and $a_{n} \neq 0$. Comparing the coefficients of $\partial^{n}$ on both sides of \eqref{eq5.28} yields $(n - 1)\lambda a_{n} = 0$, which contradicts our assumption. Therefore, $h(\partial) = a_{1}\partial + a_{0}$, where $a_{0}$, $a_{1}$ are constants. This gives
	\[
	T(L) = (a_{1}\partial^{2} + a_{0}\partial)H, \quad T(N) = -\partial H,
	\]
	Replacing $\partial H$ with $N$, we obtain part (iii). Replacing $T$ by $-T - \mathrm{id}$ yields part (iv).
\end{proof}

\begin{re}
	Observe that if both $T(L)$ and $T(N)$ belong to
	$\mathbb{C}[\partial]L+\mathbb{C}[\partial]N$, then $T$ is in fact
	a Rota--Baxter operator on $W(0)$.
	Therefore, the classification of Rota--Baxter operators on $W(0)$
	is completely determined by Theorem~\ref{thm5.3}.
\end{re}

As a consequence of Theorems \ref{thm3.14} and \ref{thm5.3}, we obtain the following result.

\begin{prop}\label{prop5.5}
	All PLCA structures on $W(0)$ are given by the following $\lambda$-products:
	\begin{itemize}
		\item[(1)]
		$L\circ_{\lambda}L=L\circ_{\lambda}H=0,\;
		H\circ_{\lambda}L=c(\partial+2\lambda)L-\lambda H,\;
		H\circ_{\lambda}H=c(\partial+\lambda)H;$
		
		\item[(2)]
		$L\circ_{\lambda}L=-(\partial+2\lambda)L,\;
		L\circ_{\lambda}H=-(\partial+\lambda)H,\;
		H\circ_{\lambda}L=-c(\partial+2\lambda)L,\;
		H\circ_{\lambda}H=-c(\partial+\lambda)H;$
		
		\item[(3)]
		$L\circ_{\lambda}L=(a_{1}\lambda^{2}-a_{0}\lambda)\lambda H,\;
		H\circ_{\lambda}L=-\lambda H,\;
		L\circ_{\lambda}H=H\circ_{\lambda}H=0;$
		
		\item[(4)]
		$L\circ_{\lambda}L=-(\partial+2\lambda)L-(a_{1}\lambda^{2}-a_{0}\lambda)\lambda H,\;
		L\circ_{\lambda}H=-(\partial+\lambda)H,\;
		H\circ_{\lambda}L=H\circ_{\lambda}H=0.$
	\end{itemize}
	where $a_{0},a_{1},c\in\mathbb{C},~c\neq0$.
\end{prop}
	To classify the non‑isomorphic PLCA structures whose underlying Lie conformal algebra is isomorphic to 
	$\bigl(W(b),[\cdot_{\lambda}\cdot]\bigr)$, we proceed case by case.  
	We treat separately the situation $b=0$ and the three families (1)–(3) for $b\neq 0$ obtained in the previous subsection.  
	In each case we write down explicitly the non‑zero $\lambda$-products that define the PLCA structure.
	
	For $b=0$ we simplify the PLCA structures listed in Proposition~\ref{prop5.5} by applying suitable basis transformations.
    
    For cases (i) and (ii) of Proposition~\ref{prop5.5}, we apply the transformations $[L' = L - \tfrac{1}{c}H,~ H' = \tfrac{1}{c}H] \quad \text{and} \quad [L' = L,~ H' = -\tfrac{1}{c}H]$,
    respectively, to obtain the unified structure:
    $$L' \circ_{\lambda} L' = -(\partial + 2\lambda)L', ~L' \circ_{\lambda} H' = -(\partial + \lambda)H', ~H' \circ_{\lambda} L' = (\partial + 2\lambda)L', ~H' \circ_{\lambda} H' = (\partial + \lambda)H'.$$
    For cases (iii) and (iv) of Proposition~\ref{prop5.5}, we consider two subcases:
    
    \noindent\textbf{Subcase 1:} $a_{1} = 0$. Applying the transformation $[L' = L + a_{0}(\partial)H,~ H' = H]$, we obtain:
    	$$L' \circ_{\lambda} L' = 0,~L' \circ_{\lambda} H' = 0,~ H' \circ_{\lambda} L' = -\lambda H',~ H' \circ_{\lambda} H' = 0$$
    for case (iii), and
    $$L' \circ_{\lambda} L' = -(\partial + 2\lambda)L',~L' \circ_{\lambda} H' = -(\partial + \lambda)H',~H' \circ_{\lambda} L' = 0,~H' \circ_{\lambda} H' = 0$$
    for case (iv).
    
    \noindent\textbf{Subcase 2:} $a_{1} \neq 0$. Applying the transformations $ [L' = L + a_{0}(\partial)H,~ H' = \tfrac{1}{a_{1}}H] \quad \text{and} \quad [L' = L + a_{0}(\partial)H,~ H' = -\tfrac{1}{a_{1}}H]$
    to cases (iii) and (iv) respectively, we obtain:
    $$L' \circ_{\lambda} L' = \lambda^{3}H',~L' \circ_{\lambda} H' = 0,~H' \circ_{\lambda} L' = -\lambda H', ~H' \circ_{\lambda} H' = 0$$
    for case (iii), and
    $$L' \circ_{\lambda} L' = -(\partial + 2\lambda)L' + \lambda^{3}H',~L' \circ_{\lambda} H' = -(\partial + \lambda)H',~H' \circ_{\lambda} L' = 0, ~H' \circ_{\lambda} H' = 0$$
    for case (iv).
    
	For $b\ne0$, in case (1), when $a=c=0$, equations \eqref{eq5.17} and \eqref{eq5.18} yield two non-isomorphic PLCA structures:
	\begin{align*}
		&L\circ_{\l}L=L\circ_{\l}H=H\circ_{\l}H=0,~~H\circ_{\l}L=-\p H; \\ 
		&L\circ_{\l}L=-(\p+2\l)L,~~L\circ_{\l}H=-\p H,~~H\circ_{\l}L=H\circ_{\l}H=0. 
	\end{align*}
	For $c\ne0$, let $k=\frac{a}{c}$. The change of basis $L' = L$, $H' = cH$ transforms \eqref{eq5.17} and \eqref{eq5.18} into the following two non-isomorphic PLCA structures, respectively:
	\begin{align*}
		&L'\circ_{\l}L'=-\p(k\l^{2}-\l)H',~~H'\circ_{\l}L'=-\p H',~~L'\circ_{\l}H'=H'\circ_{\l}H'=0; \\ 
		&L'\circ_{\l}L'=-(\p+2\l)L'+\p(k \l^{2}- \l)H',~~L'\circ_{\l}H'=-\p H',~~H'\circ_{\l}L'=H'\circ_{\l}H'=0. 
	\end{align*}
	If $c=0$ and $a\ne0$, we perform the basis transformation $[L'=L,~H'=aH]$. Then equations \eqref{eq5.17} and \eqref{eq5.18} reduce to the following two non-isomorphic PLCA structures, respectively:
	\begin{align*}
		&L'\circ_{\l}L'=-\p\l^{2}H',~~H'\circ_{\l}L'=-\p H',~~L'\circ_{\l}H'=H'\circ_{\l}H'=0; \\ 
		&L'\circ_{\l}L'=-(\p+2\l)L'+\p\l^{2}H',~~L'\circ_{\l}H'=-\p H',~~H'\circ_{\l}L'=H'\circ_{\l}H'=0. 
	\end{align*}
	
	In case (2), when $a=c=0$, equations \eqref{eq5.19} and \eqref{eq5.20} give the following two non-isomorphic PLCA structures, respectively:
	\begin{align*}
		&H\circ_{\l}L=-(2\p+\l)H,~~L\circ_{\l}L=L\circ_{\l}H=H\circ_{\l}H=0; \\ 
		&L\circ_{\l}L=-(\p+2\l)L,~~L\circ_{\l}H=-(\p-\l)H,~~H\circ_{\l}L=H\circ_{\l}H=0.
	\end{align*}
	For $c\ne0$, let $k=\frac{a}{c}$. We perform the basis transformation $[L'=L,~H'=cH]$, then \eqref{eq5.19} and \eqref{eq5.20} are transformed into the following two non-isomorphic PLCA structures, respectively:
	\begin{align*}
		&L'\circ_{\l}L'=(k \l^{3}+ \l)(2\p+\l)H',~~L'\circ_{\l}H'=0,~~H'\circ_{\l}L'=-(2\p+\l)H',~~H'\circ_{\l}H'=0; \\ 
		&L'\circ_{\l}L'=-(\p+2\l)L'-(k\l^{3}+\l)(2\p+\l)H',~~L'\circ_{\l}H'=-(\p-\l)H',~~H'\circ_{\l}L'=0,~~H'\circ_{\l}H'=0. 
	\end{align*}
	If $c=0$ and $a\ne0$, applying the basis transformation $[L'=L,~H'=aH]$ to equations \eqref{eq5.19} and \eqref{eq5.20}  yields two non-isomorphic PLCA structures:
	\begin{align*}
		&L'\circ_{\l}L'=\l^{3}(2\p+\l)H',~~L'\circ_{\l}H'=0,~~H'\circ_{\l}L'=-(2\p+\l)H',~~H'\circ_{\l}H'=0; \\ 
		&L'\circ_{\l}L'=-(\p+2\l)L'-\l^{3}(2\p+\l)H',~~L'\circ_{\l}H'=-(\p-\l)H',~~H'\circ_{\l}L'=H'\circ_{\l}H'=0. 
	\end{align*}
	
	We now analyze case (3).  When $c=0$, equations \eqref{eq5.21} and \eqref{eq5.22} directly give two non-isomorphic PLCA structures:
	\begin{align*}
		&H\circ_{\l}L=(-b\p+(1-b)\l)H,~~L\circ_{\l}L=L\circ_{\l}H=H\circ_{\l}H=0; \\ &L\circ_{\l}L=-(\p+2\l)L,~~L\circ_{\l}H=-(\p+(1-b)\l)H,~~H\circ_{\l}L=H\circ_{\l}H=0. 
	\end{align*}
	For $c\neq0$, we perform the basis transformation $[L'=L-\frac{1}{c}H,~H'=\frac{1}{c}H]$. This transforms \eqref{eq5.21} and \eqref{eq5.22} into the following non-isomorphic PLCA structures:
	\begin{align*}
		&L'\circ_{\l}L'=-\l(-b\p+(1-b)\l)H',~~L'\circ_{\l}H'=0,~~H'\circ_{\l}L'=(-b\p+(1-b)\l)H',~~H'\circ_{\l}H'=0; \\
		&L'\circ_{\l}L'=-(\p+2\l)L'+\l(-b\p+(1-b)\l)H',~~L'\circ_{\l}H'=-(\p+(1-b)\l)H',~~H'\circ_{\l}L'=H'\circ_{\l}H'=0.
	\end{align*}
	
	We summarize the discussions above in the following theorem:
	\begin{theo}
		Let \big($W(b), \circ_{\l}, [\cdot_\l\cdot]\big)$ be a non-trival PLCA structure on the Lie conformal algebra $W(b)$, where the $\l$-bracket is defined in \eqref{eq5.1}. Then 
		\begin{itemize}
			\item [(1)] For $b=0$, up to isomorphism, the $\lambda$-product $\circ_\lambda$ is given by one of the following cases:
			\begin{enumerate}[label=(\roman*)]
				\item [(i)]$L \circ_{\lambda} L = -(\partial + 2\lambda)L,\quad
					L \circ_{\lambda} H = -(\partial + \lambda)H,\quad 
					H \circ_{\lambda} L = (\partial + 2\lambda)L,\quad 
					H \circ_{\lambda} H = (\partial + \lambda)H.$	
				\item[(ii)] $	L \circ_{\lambda} L = 0,\quad 
					L \circ_{\lambda} H = 0,\quad 
					H \circ_{\lambda} L = -\lambda H,\quad 
					H \circ_{\lambda} H = 0.$
				\item[(iii)] $	L \circ_{\lambda} L = -(\partial + 2\lambda)L,\quad 
					L \circ_{\lambda} H = -(\partial + \lambda)H,\quad 
					H \circ_{\lambda} L = 0,\quad 
					H \circ_{\lambda} H = 0.$
				\item [(iv)]$	L \circ_{\lambda} L = \lambda^{3}H,\quad 
					L \circ_{\lambda} H = 0,\quad 
					H \circ_{\lambda} L = -\lambda H,\quad 
					H \circ_{\lambda} H = 0.$
				\item[(v)] $	L \circ_{\lambda} L = -(\partial + 2\lambda)L + \lambda^{3}H,\quad 
					L \circ_{\lambda} H = -(\partial + \lambda)H, \quad
					H \circ_{\lambda} L = 0, \quad
					H \circ_{\lambda} H = 0.$
			\end{enumerate}
			\item [(2)] For $b=1$, up to isomorphism, the $\lambda$-product $\circ_\lambda$ is given by one of the following cases:
			\begin{enumerate}[label=(\roman*)]
				\item [(i)] $L\circ_{\l}L=L\circ_{\l}H=H\circ_{\l}H=0,\quad H\circ_{\l}L=-\p H.$
				\item [(ii)] $L\circ_{\l}L=-\p\l^{2}H,\quad L\circ_{\l}H=0,\quad H\circ_{\l}L=-\p H,\quad H\circ_{\l}H=0.$
				\item [(iii)] $L\circ_{\l}L=-(\p+2\l)L,\quad L\circ_{\l}H=-\p H,\quad H\circ_{\l}L=0,\quad H\circ_{\l}H=0.$		
				\item [(iv)] $L\circ_{\l}L=-(\p+2\l)L+\p\l^{2}H,\quad L\circ_{\l}H=-\p H,\quad H\circ_{\l}L=0,\quad H\circ_{\l}H=0.$
				\item [(v)] $L\circ_{\l}L=-\p(k\l^{2}-\l)H,\quad L\circ_{\l}H=0,\quad H\circ_{\l}L=-\p H,\quad H\circ_{\l}H=0,\quad k\in\C.$
				\item [(vi)] $L\circ_{\l}L=-(\p+2\l)L+\p(k \l^{2}-\l)H,\quad L\circ_{\l}H=-\p H,\quad H\circ_{\l}L=0,\quad H\circ_{\l}H=0,\quad k\in\C.$		
			\end{enumerate}		
			
			\item [(3)] For $b=2$, up to isomorphism, the $\lambda$-product $\circ_\lambda$ is given by one of the following cases:
			\begin{enumerate}[label=(\roman*)]
				\item [(i)] $L\circ_{\l}L=0,\quad L\circ_{\l}H=0,\quad H\circ_{\l}L=-(2\p+\l)H,\quad H\circ_{\l}H=0.$
				\item [(ii)] $L\circ_{\l}L=-(\p+2\l)L,\quad L\circ_{\l}H=-(\p-\l)H,\quad H\circ_{\l}L=0,\quad H\circ_{\l}H=0.$
				\item [(iii)] $L\circ_{\l}L=\l^{3}(2\p+\l)H,\quad L\circ_{\l}H=0,\quad H\circ_{\l}L=-(2\p+\l)H,\quad H\circ_{\l}H=0.$
				\item [(iv)] $L\circ_{\l}L=-(\p+2\l)L-\l^{3}(2\p+\l)H,\quad L\circ_{\l}H=-(\p-\l)H,\quad H\circ_{\l}L=0,\quad H\circ_{\l}H=0.$
				\item [(v)] $L\circ_{\l}L=(k \l^{3}+ \l)(2\p+\l)H,\quad L\circ_{\l}H=0,\quad H\circ_{\l}L=-(2\p+\l)H,\quad H\circ_{\l}H=0,\quad k\in\C.$
				\item [(vi)] $L\circ_{\l}L=-(\p+2\l)L-(k\l^{3}+\l)(2\p+\l)H,~ L\circ_{\l}H=-(\p-\l)H,~ H\circ_{\l}L=0,~ H\circ_{\l}H=0,~k\in\C.$
				
			\end{enumerate}		
			
			\item [(4)] For $b\ne0,1,2$, up to isomorphism, the $\lambda$-product $\circ_\lambda$ is given by one of the following cases: 
			\begin{enumerate}[label=(\roman*)]
				
				\item [(i)] $L\circ_{\l}L=0,\quad L\circ_{\l}H=0,\quad H\circ_{\l}L=(-b\p+(1-b)\l)H,\quad H\circ_{\l}H=0$
				\item [(ii)] $L\circ_{\l}L=-(\p+2\l)L,\quad L\circ_{\l}H=-(\p+(1-b)\l)H,\quad H\circ_{\l}L=0,\quad H\circ_{\l}H=0.$    	
				\item [(iii)] $L\circ_{\l}L=-\l(-b\p+(1-b)\l)H,\quad L\circ_{\l}H=0,\quad H\circ_{\l}L=(-b\p+(1-b)\l)H,\quad H\circ_{\l}H=0$		
				\item [(vi)] $L\circ_{\l}L=-(\p+2\l)L+\l(-b\p+(1-b)\l)H,\quad L\circ_{\l}H=-(\p+(1-b)\l)H,\quad H\circ_{\l}L=0,\quad H\circ_{\l}H=0.$
			\end{enumerate}		
		\end{itemize}
	\end{theo}

\end{document}